\documentclass[11pt]{article}
\usepackage[margin=1in]{geometry}
\usepackage{amsmath,amssymb,amsthm,mathtools}
\usepackage{microtype}
\usepackage{enumitem}
\usepackage{hyperref}
\hypersetup{hidelinks}

\newtheorem{theorem}{Theorem}[section]
\newtheorem{lemma}[theorem]{Lemma}
\newtheorem{claim}[theorem]{Claim}
\newtheorem{remark}[theorem]{Remark}
\newtheorem{corollary}[theorem]{Corollary}
\theoremstyle{definition}
\newtheorem{definition}[theorem]{Definition}
\usepackage{tikz}
\usepackage{xcolor}
\usepackage{graphicx}
\usetikzlibrary{calc,decorations.pathreplacing,fit,positioning}
\newcommand{\ol}[1]{\overline{#1}}
\usetikzlibrary{arrows.meta,calc,decorations.pathreplacing,decorations.markings,backgrounds}

\begin{document}

\title{An Optimal Bound for Ramsey Goodness of Cycles}
\author{Peiru Kuang\footnote{School of Mathematical Sciences, Shanghai Jiao Tong University, Shanghai 200240, China. Email: peiru\_k@sjtu.edu.cn}
\and
Yan Wang\footnote{School of Mathematical Sciences, Shanghai Jiao Tong University, Shanghai 200240, China. Supported by National Key R\&D Program of China under Grant No. 2022YFA1006400 and National Natural Science Foundation of China under Grant No. 12571376. Email: yan.w@sjtu.edu.cn (corresponding author).}}
\date{}
\maketitle
\begin{abstract}
For graphs $F$ and $H$, the Ramsey number $R(F,H)$ is the minimum integer $N$ such that every $N$-vertex graph contains $F$ or its complement contains $H$. Let $\sigma(H)$ denote the minimum order of a color class in a proper $\chi(H)$-coloring of $H$. If $F$ is connected and $|F|\ge\sigma(H)$, a construction of Burr gives $R(F,H)\ge(\chi(H)-1)(|F|-1)+\sigma(H)$. Burr proved that this bound is attained for $F=C_n$ when $n$ is sufficiently large. Allen, Brightwell and Skokan conjectured that equality holds whenever $n\geq |H| \chi(H)$, while Haslegrave, Hyde, Kim and Liu subsequently proved it whenever $n\ge C|H|\log^4\chi(H)$.

Pokrovskiy and Sudakov conjectured that the optimal linear condition $n\geq C|H|$ suffices; this conjecture was also highlighted by Montgomery in his 2026 ICM survey (see Conjecture 9.2). In this paper, we resolve this conjecture by proving that there is an absolute constant $C>0$ such that $R(C_n,H)=(\chi(H)-1)(n-1)+\sigma(H)$ for every nonempty graph $H$ and every $n\ge C|H|$. This gives the first bound linear in $|H|$, and is best possible up to a constant factor. Our proof builds on the framework of Haslegrave, Hyde, Kim and Liu, and combines some new ideas in expansion and switching cycle lengths.
\end{abstract}

\section{Introduction}
For graphs \(F\) and \(H\), a celebrated result of Ramsey~\cite{Ramsey} from 1930 says that there exists an integer \(R(F,H)\) such that every graph \(G\) on \(R(F,H)\) vertices either contains a copy of \(F\), or its complement \(\ol G\) contains a copy of \(H\). Ramsey theory is a central topic of research in combinatorics, with rich applications in number theory, information theory and theoretical computer science, etc. Determining the exact value of \(R(F,H)\) is notoriously difficult even for very simple graphs; indeed, \(R(K_5,K_5)\) is still unknown; see the survey of Conlon, Fox and Sudakov~\cite{CFS} for more examples.

Ramsey theory has witnessed several major breakthroughs in recent years,
even for the classical problem of determining clique Ramsey numbers.
Campos, Griffiths, Morris and Sahasrabudhe~\cite{CGMS} obtained the first
exponential improvement over the Erd\H{o}s--Szekeres upper bound for the
diagonal Ramsey numbers, proving that
\(R(K_k,K_k)\le(4-\varepsilon)^k\) for some \(\varepsilon>0\).
Balister, Bollob\'as, Campos, Griffiths, Hurley, Morris, Sahasrabudhe and
Tiba~\cite{BBCGHMST} subsequently obtained an exponential improvement in
the multicolor setting. On the lower-bound side, Ma, Shen and
Xie~\cite{MSX} obtained the first exponential improvement over the
classical probabilistic lower bound of Erd\H{o}s for
\(R(K_\ell,K_{C\ell})\) for every fixed \(C>1\).
More recently, Brada\v{c}~\cite{Bradac} proved that
\(R(K_s,K_k)\ge k^{s-2+o(1)}\) for every fixed \(s\ge4\), and more
generally determined the exponent asymptotically in a broad off-diagonal
range. See also the recent survey of Morris in his 2026 ICM lecture~\cite{MorrisICM}.

Let \(\chi(H)\) be the chromatic number of \(H\), and let
\(\sigma(H)\) be the minimum order of a color class in a proper
\(\chi(H)\)-coloring of \(H\). Building on observations of
Erd\H{o}s~\cite{Erdos1947} and Chv\'atal and Harary~\cite{CH},
Burr~\cite{Burr} proved that, whenever \(F\) is connected and
\(|F|\ge\sigma(H)\),
\begin{equation}\label{eq:burr-intro}
R(F,H)\ge(\chi(H)-1)(|F|-1)+\sigma(H).
\end{equation}
His construction consists of \(\chi(H)-1\) disjoint cliques of order
\(|F|-1\) and one disjoint clique of order \(\sigma(H)-1\). 
Following Burr and Erd\H{o}s~\cite{BurrErdos}, a graph \(F\) is called
\emph{\(H\)-good} if equality holds in \eqref{eq:burr-intro}.

Several classical exact Ramsey results may be viewed as goodness
statements. Erd\H{o}s~\cite{Erdos1947} proved
that paths are good versus complete graphs, while
Gerencs\'er and Gy\'arf\'as~\cite{GG} determined the Ramsey numbers of
two paths. Chv\'atal~\cite{Chvatal} then extended this to all trees, proving that every tree is good versus every complete graph. These results suggested that \eqref{eq:burr-intro} is often attained when the first graph is sparse. Later, Burr and Erd\H{o}s~\cite{BurrErdos} initiated the systematic study of Ramsey goodness. 
A general structural result was
obtained by Allen, Brightwell and Skokan~\cite{ABS}: For every fixed
graph \(H\) and integer \(\Delta\), every sufficiently large connected
graph \(F\) with maximum degree at most \(\Delta\) and sufficiently
small bandwidth is \(H\)-good. 
Moreover, Chv\'atal, R\"odl, Szemer\'edi and Trotter~\cite{CRST} established linear Ramsey bounds for graphs of bounded maximum degree. Lee~\cite{Lee} later extended this result to
graphs of bounded degeneracy, thereby resolving a longstanding conjecture of Burr and Erd\H{o}s~\cite{BurrErdosMagnitude}.

Besides these general structural results, sharper quantitative and
exact results were obtained for some families of sparse and
structured graphs. Erd\H{o}s, Faudree, Rousseau and
Schelp~\cite{EFRS85} obtained results on Ramsey numbers involving
complete multipartite graphs versus sparse graphs. Pokrovskiy and
Sudakov~\cite{PSPaths} proved that \(P_n\) is \(H\)-good whenever
\(n\ge4|H|\). Balla, Pokrovskiy and Sudakov~\cite{BPS} extended this to bounded-degree trees. More precisely, for every
fixed \(\Delta\) and \(k\), they proved that there exists a constant
\(C_{\Delta,k}\) such that every \(n\)-vertex tree \(T\) with
\(\Delta(T)\le\Delta\) is \(H\)-good for every graph \(H\) with
\(\chi(H)=k\), provided
$
 n\ge C_{\Delta,k}|H|\log^4|H|.
$
For fixed \(\Delta\) and \(\chi(H)\), Montgomery, Pavez-Sign\'e and
Yan~\cite{MPY} subsequently removed the logarithmic factor and obtained
a linear bound. Further results on Ramsey numbers for sparse and
structured graphs may be found in \cite{CFLS,FPGMSS,FHW,NikiforovRousseau}.

The case in which the first graph is a cycle has a particularly rich
history. Bondy and Erd\H{o}s~\cite{BondyErdos}, Faudree and
Schelp~\cite{FaudreeSchelp}, and Rosta~\cite{RostaI,RostaII} determined the
Ramsey numbers of two cycles. In particular,
$
R(C_n,C_k)=2n-1
$
whenever \(n\ge k\ge3\), \(k\) is odd, and
\((n,k)\ne(3,3)\). Since \(\chi(C_k)=3\) and \(\sigma(C_k)=1\) for
odd \(k\), this is precisely the statement that \(C_n\) is
\(C_k\)-good.
The corresponding problem for cycles versus complete graphs was initiated by Bondy
and Erd\H{o}s~\cite{BondyErdos}, who proved that \(C_n\) is
\(K_k\)-good whenever \(n\ge k^2-2\). More exact values were
subsequently determined for small complete graphs. Erd\H{o}s,
Faudree, Rousseau and Schelp~\cite{EFRS} conjectured that the natural
condition \(n\ge k\ge3\) already suffices. Nikiforov~\cite{Nikiforov}
proved that \(R(C_n,K_k)=(n-1)(k-1)+1\) whenever \(n\ge4k+2\). Keevash, Long and
Skokan~\cite{KLS} later showed that, for large \(k\), it holds under
the substantially weaker condition
\(n\ge C\log k/\log\log k\), where \(C\) is an absolute constant.
They also showed that this range is best possible up to the value of
\(C\).

These results lead naturally to the problem of determining when
\(C_n\) is $H$-good for an arbitrary graph \(H\).
Burr~\cite{Burr} showed that, for every fixed graph \(H\), the cycle \(C_n\) is \(H\)-good for all sufficiently large \(n\). Allen, Brightwell and Skokan~\cite{ABS} conjectured the explicit bound
\(n\ge\chi(H)|H|\). Pokrovskiy and Sudakov~\cite{PS} obtained a linear
bound under the additional assumption
\(\sigma(H)\ge\chi(H)^{22}\), and conjectured that there is an absolute
constant \(C\) such that \(C_n\) is \(H\)-good whenever
\(n\ge C|H|\). This conjecture was also highlighted by Montgomery in
his 2026 ICM survey (see Conjecture 9.2 in~\cite{MontgomeryICM}). Haslegrave, Hyde, Kim and Liu~\cite{HHKL} subsequently proved that \(C_n\) is \(H\)-good for
every graph \(H\) whenever
$
n\ge C|H|\log^4\chi(H).
$
They further asked whether the logarithmic
factor could be removed.
%They also mentioned that it would be interesting to at least get rid of the logarithmic factor and obtain a bound linear in $|H|$.

In this paper, we prove the conjecture of Pokrovskiy and Sudakov.
\begin{theorem}\label{thm:main}
There exists an absolute constant $C>0$ such that, for every nonempty graph
$H$ and every integer $n\ge C|H|$, the cycle $C_n$ is $H$-good; that is,
$
 R(C_n,H)=(\chi(H)-1)(n-1)+\sigma(H).
$
\end{theorem}

%It is enough to prove the upper bound when \(H\) is complete multipartite. Indeed, take a proper \(\chi(H)\)-coloring whose smallest class has order \(\sigma(H)\), and add all missing edges between distinct color classes. This preserves both \(|H|\) and \(\sigma(H)\).

The linear bound on $|H|$ in Theorem~\ref{thm:main} is necessary.
Haslegrave, Hyde, Kim and Liu~\cite{HHKL} constructed complete
multipartite graphs \(H\) for which \(C_n\) is not \(H\)-good when
\(n=(1-o(1))|H|\). The following example gives a stronger 
lower bound for constant \(C \ge 2-o(1)\). Let \(t\ge2\), set \(H:=K_{1,t}\) and
\(n:=2t-1\), and consider \(G:=K_{t,t}\). Since \(G\) is bipartite and
\(n\) is odd, it is \(C_n\)-free. Moreover,
\(\ol G=K_t\mathbin{\dot\cup}K_t\), which is \(H\)-free. Consequently,
$
 R(C_n,H)\ge2t+1>
 (\chi(H)-1)(n-1)+\sigma(H)=n.
$
Thus \(C_n\) is not \(H\)-good, while
\(n/|H|=(2t-1)/(t+1)\) tends to \(2\) when $t$ tends to infinity.

\subsection{Overview}

It is enough to consider a complete multipartite graph
\(H=K_{m_1,\ldots,m_k}\), and we argue by induction on \(k\), namely the number of its parts. Write \(h:=|H|\). Finding a cycle longer than \(n\) is relatively easy; the main difficulty is to obtain a cycle of exactly the prescribed length \(n\).

Our proof follows the framework of Haslegrave, Hyde, Kim and
Liu~\cite{HHKL,LMOdd}, and Pokrovskiy and Sudakov~\cite{PS}. The main idea is to control the length in two ways. First, we
construct a structure whose length can be varied over an interval of \(\Theta(h)\) consecutive integers; we refer to this as \emph{fine switching}. Secondly, we shorten a long common subpath of the structure without affecting the range of attainable lengths; we refer to this operation as \emph{free-path shortening}. We use the second operation to bring the length close to \(n\), and then the first operation to obtain length exactly \(n\).

\medskip
\noindent\textbf{Fine switching.}
A \emph{protected switcher} is, informally, a cycle-like structure with several local choices: At each switch, we may use either of two paths with the same endvertices but different lengths. A particular choice at every such place gives a \emph{route} through the switcher. The \emph{flexibility} of the switcher measures how much the length of a route can be changed by making these choices. We construct the switcher so that every integer length in the resulting interval is attained; see Lemma~\ref{lem:interval-routes}. The switches are grouped into protected blocks so that, when the switcher is later joined to other cycles, only a few blocks are affected and most of the flexibility is preserved.

The main difficulty is to construct such a switcher with
\(\Theta(h)\) flexibility using only \(O(h)\) vertices. We build on the sublinear-expansion approach of Haslegrave, Hyde, Kim and Liu~\cite{HHKL}, a method originating in the work of Koml\'os and Szemer\'edi~\cite{KS1,KS2} and subsequently developed and applied in many extremal problems; see, for example, \cite{FriedmanKrivelevich,KrivelevichLocallySparse,KrivelevichExp,KrivelevichLongCycles,Letzter,LMMader,MontgomeryMinors,MontgomeryICM}. We use a variable-rate form of expansion, where the required rate depends on the
size of the set and the number of parts remaining in the forbidden multipartite graph.

The key point is that, when expansion fails, the resulting sparse cut can be used to reduce the problem to fewer parts. Indeed, if a set \(S\) has few neighbors, then \(S\) is complete to \(V(G)\setminus(S\cup N_G(S))\) in the complement of $G$. Hence, if some parts of the forbidden multipartite graph occur on one side, the other side must avoid the multipartite graph formed by the remaining parts. Consequently, either we reduce the problem to fewer parts, or we obtain an expanding induced subgraph; this dichotomy is formalized in Lemma~\ref{lem:expanding-subgraph}.

After obtaining an expanding subgraph, we use it to construct the protected switcher for fine switching. Short odd cycles provide the basic switches: Two almost opposite vertices determine two paths whose lengths differ by one, while expansion remains after deleting the cycle; see Lemma~\ref{lem:compact-odd}. To connect the switches disjointly, we use rooted sets, in which every vertex is joined to a prescribed vertex by a short internal path. We choose these rooted sets pairwise disjoint using Lemma~\ref{lem:multi-root}, and connect them by short paths
avoiding previously used vertices using Lemma~\ref{lem:short-path}. Iterating this procedure and grouping the switches into protected blocks yields a protected switcher on \(O(h)\) vertices with \(\Theta(h)\) flexibility, as established in Lemma~\ref{lem:flexible-switcher}.

\medskip
\noindent\textbf{Free-path shortening.}
We use a \emph{free subpath}, meaning a subpath which is contained in every route and is disjoint from all protected blocks. Thus shortening a free subpath shortens every route in exactly the same way, while leaving all switches and hence the flexibility unchanged.

Starting from the protected switcher, we repeatedly join it to long cycles obtained from Lemma~\ref{lem:long-cycle}. Each joining preserves most of the flexibility and makes almost all of the new cycle free; see Lemma~\ref{lem:protected-splice}. After a bounded number of joinings, the longest route exceeds \(n\) and contains enough free length. Lemma~\ref{lem:shortcut} then allows us to shorten free subpaths without changing any switch, until the longest route is within \(O(h)\) of \(n\). Since \(\Theta(h)\) flexibility remains, Lemma~\ref{lem:interval-routes} gives a route of length exactly \(n\).

\medskip
The above argument also explains the main difference from the proof of Haslegrave, Hyde, Kim and Liu~\cite{HHKL}. Their adjusters, which may be viewed as a special case of our switcher in which every switch has
weight one, provide only \(\Omega(h/\log^2 k)\) flexibility, so \(O(\log^2 k)\) of them must be combined to obtain linear flexibility. Moreover, with \(Q\) vertex-disjoint connecting paths, their joining procedure loses an
\(O(Q^{-1/2})\) proportion of the flexibility at each step, leading to the requirement \(Q=\Omega(\log^4 k)\). In contrast, our variable-rate expansion produces a single switcher with \(\Theta(h)\) flexibility, while the protected blocks reduce the proportional loss in each
joining step to \(O(Q^{-1})\), apart from the loss of flexibility from only a bounded number of protected blocks. Since only a bounded number of joinings are needed, a constant number of vertex-disjoint connecting paths suffices, removing the \(\log^4 k\) factor.

The final part of the proof is a stability argument describing the global structure of \(G\). Either the problem reduces to a large induced subgraph whose complement avoids a multipartite graph with one
fewer part, or \(G\) contains \(k-1\) large pairwise anticomplete sets with strong internal connectivity; see Lemma~\ref{lem:stability}. The first case is handled by induction or Corollary~\ref{cor:PS-large}. In
the second case, we extract large \(2\)-connected expanding subgraphs from these sets. If two can be linked by two vertex-disjoint paths, the protected-switcher argument above gives \(C_n\); otherwise, the block structure yields a cutvertex separation, and Lemma~\ref{lem:PS-path} gives \(C_n\).

\medskip
\noindent\textbf{Organization.}
Section~2 contains notation and some results used later. 
The two main ingredients, the variable-rate expander and the switcher structure with related lemmas, are given in Sections 3 and 4, respectively. Section~5 constructs protected switchers with \(\Theta(h)\) flexibility
and establishes the local embedding lemma used in the stability argument. Section~6 establishes
the stability result and proves the main theorem.
\section{Preliminaries}

All graphs are finite and simple. For a graph \(G\), its complement
is denoted by \(\ol G\), and we write \(|G|:=|V(G)|\). For $A\subseteq V(G)$, let
$
 N_G(A)=\{v\in V(G)\setminus A: uv\in E(G)\text{ for some }u\in A\}
$
be the external neighborhood of $A$ in $G$.
We write $G[A]$ for the subgraph induced by $A$ and $G-A$ for
$G[V(G)\setminus A]$. The length of a path or cycle is its number of edges. All logarithms are natural. Write $[t]=\{1,\ldots,t\}$. A complete multipartite graph with part sizes $a_1,\ldots,a_q$ is denoted by $K_{a_1,\ldots,a_q}$, and a graph induced by a subcollection of its parts is a \emph{part-subgraph}. 
For disjoint sets \(A,B\subseteq V(G)\), let \(e_G(A,B)\) denote
the number of edges between \(A\) and \(B\) in $G$. Given graphs \(G\) and \(H\), we call \(G\) \emph{\(H\)-free} if
\(G\) does not contain \(H\) as a subgraph.
%This is the notion denoted by $H'\preceq H$ in Section~3 of~\cite{HHKL}.

We will use the following results. The first is Theorem~1.2 in \cite{HHKL} of Haslegrave, Hyde, Kim and Liu.

\begin{theorem}[Haslegrave--Hyde--Kim--Liu~\cite{HHKL}]
\label{thm:HHKL}
There exists an absolute constant $C>0$ such that $C_n$ is $H$-good for
every graph $H$ and every $n\ge C|H|\log^4\chi(H)$.
\end{theorem}

The next corollary is Corollary~3.2 in \cite{HHKL}, which is deduced from Pokrovskiy and Sudakov~\cite{PS}. It gives the desired exact Ramsey number when the
largest part of \(H\) has order at least \(k^{22}\).
\begin{corollary}[Haslegrave--Hyde--Kim--Liu~\cite{HHKL}]
\label{cor:PS-large}
Let $H=K_{m_1,\ldots,m_k}$, where $m_1\le\cdots\le m_k$ and
$m_k\ge k^{22}$.  If
$
 n\ge 10^{60}m_k,
$
then $R(C_n,H)=(k-1)(n-1)+m_1$.
\end{corollary}

The following bipartite statement is an extension of Corollary~3.8 in~\cite{PS} for all $b$.
\begin{corollary}
\label{cor:PS-bip}
Let $b\ge a\ge1$ and let $b_0=\max\{b,8\}$.  If
$
 n\ge 2\times10^{49}b_0,
$
then $R(C_n,K_{a,b})=n+a-1$.
\end{corollary}

\begin{proof}
For $b\ge8$, this is exactly Corollary~3.8 in~\cite{PS}. For $b<8$,
apply that corollary to $K_{a,8}$.
\end{proof}

The path lemma is Lemma 3.7 in \cite{PS} with the extension to $m<8$. It shows that, if the complement is
\(K_{m,m}\)-free and \(|A\cup N_G(A)|\ge n\) for every set \(A\) of
at least \(m\) vertices, then any sufficiently long \(x\)--\(y\) path
can be replaced by an \(x\)--\(y\) path on exactly \(n\) vertices.
\begin{lemma}
\label{lem:PS-path}
Let $m\ge1$, write $\bar m=\max\{m,8\}$, and suppose that
$
 n\ge 2\times10^{49}\bar m.
$
Let $G$ be a graph such that $\ol G$ is $K_{m,m}$-free and
$
 |A\cup N_G(A)|\ge n
$
for every $A\subseteq V(G)$ with $|A|\ge m$. If $x,y\in V(G)$ are joined by
an $x$--$y$ path on at least $8\bar m$ vertices, then they are joined by an $x$--$y$ path on exactly $n$ vertices.
\end{lemma}

\begin{proof}
For $m\ge8$, this is Lemma 3.7 in \cite{PS}. If $m<8$, apply that lemma with parameter \(8\), since its assumptions are immediate from those above.
\end{proof}

\section{Variable-rate expansion}

Given a graph whose complement does not contain a prescribed complete multipartite graph, we shall pass to an induced subgraph which is not too large and has suitable expansion properties. The expansion rate depends on both the size of the set and the number of parts in the current forbidden part-subgraph.

Throughout this section, \(r\le q\) denotes the number of parts in the current forbidden part-subgraph, while \(q\) remains unchanged when \(r\) decreases.
Fix \(\eta:=10^{-4}\) and \(K:=10^{4}\). For integers \(d\ge1\), real \(\beta\ge1\), and \(q\ge2\), let \(p_{1,q}(d,\beta):=2d\), and, for
\(2\le r\le q\), let
\[
 p_{r,q}(d,\beta):=
 2d+\left\lceil
 K\beta\left(
 2r-3+\frac{(r-1)\log r}{\log(2q)}
 \right)
 \right\rceil.
\]
When \(d\), \(\beta\), and \(q\) are fixed, we write \(p_r\) for
\(p_{r,q}(d,\beta)\). When \(\beta\ge qd\), directly from the definition,
$
\frac12K\beta u\le p_{u,q}(d,\beta)\le4K\beta u
$
for every $2\le u\le q$.
We will use this inequality many times in the following proof.

For \(2\le r\le q\) and \(x\ge0\), let \(s\) be the largest integer in \([r-1]\) such that either \(s=1\) or \(p_{s,q}(d,\beta)\le x\), and define
\[
 \phi_{r,q,\beta}(x):=
 \frac{ \eta\left(1+\log(r/s)\right)}{\log(2q)}.
\]
We use a variable-rate extension of the expansion condition in Definition~4.1
\cite{HHKL}, which builds on the sublinear-expansion
methods of Liu and Montgomery~\cite{LMOdd}. The expansion rate for
large vertex sets is allowed to depend on the size of the set and on
the number of parts in the current forbidden part-subgraph.
\begin{definition}\label{def:compact-expansion}
A graph \(F\) \textit{multi-\((d,r,q,\beta,c)\)-expands into \(W\subseteq V(F)\)} if the following hold for every $S\subseteq V(F)$:
\begin{enumerate}[label=\textup{(\roman*)}]
\item \(|N_F(S)\cap W|\ge c|S|\) whenever \(|S|\le\beta\);
\item \(|N_F(S)|\ge
\phi_{r,q,\beta}(|S|)|S|+10\beta\) whenever
\(\beta\le|S|\le|F|/2\).
\end{enumerate}
%It \emph{strongly multi-\((d,r,q,\beta)\)-expands} if \(16\) in \textup{(i)} is replaced by \(20\). 
When \(W=V(F)\), we omit the words ``into \(W\)''.
\end{definition}

The next lemma helps us find an expander subgraph while preserving a complete multipartite forbidden graph with at least two parts.

\begin{lemma}\label{lem:expanding-subgraph}
Let \(\mathcal J\) be a complete \(q\)-partite graph whose parts have size in \([d,2d]\), and let \(J\subseteq\mathcal J\) be induced by \(r\ge2\) of its parts. Suppose that \(\beta\ge qd\),
\(\ol G\) is \(J\)-free, and \(|G|\ge p_{r,q}(d,\beta)\). Then there are
a part-subgraph \(J'\subseteq J\) with \(r'\ge2\) parts and an induced subgraph \(F\subseteq G\) such that \(\ol F\) is \(J'\)-free, \(p_{r',q}(d,\beta)-2d\le|F|\le p_{r',q}(d,\beta)\), and \(F\) multi-\((d,r',q,\beta,20)\)-expands.
\end{lemma}

\begin{proof}
Fix \(d,\beta\), and \(q\). We proceed by induction on \(r\). 
Call a set \(S\subseteq V(G)\) \textit{bad} if \(2d\le|S|\le p_r/2\) and
$
 |N_G(S)|\le
 \phi_{r,q,\beta}(|S|)|S|+200\beta.
$
The argument below also covers the case when \(r=2\): In this case, a bad set gives a copy of \(J\) in \(\ol G\), while if no bad set exists, the required subgraph is obtained directly.
By removing vertices, we may assume that \(|G|=p_{r,q}(d,\beta)\).

\medskip
\noindent \textbf{Case 1.}
A bad set \(S\) exists.  
\medskip

Let \(x:=|S|\) and \(T:=V(G)\setminus(S\cup N_G(S))\). Let \(s\) be the integer used in the definition of \(\phi_{r,q,\beta}(x)\). 
%Note that in this case, we have $r \ge 3$.
\begin{claim}
\(|T|\ge p_{\lceil r/2\rceil}\).
\end{claim}
Note that \(\phi_{r,q,\beta}\leq 3\eta\). If \(x\le p_r/3\), then
$
|T|\ge\left(\frac23-\eta\right)p_r-200\beta
 \ge p_{\lceil r/2\rceil}
$.
Now suppose that \(x>p_r/3\). 
Thus,
$
p_{s+1}>x>\frac{p_r}3.
$
Hence
$
4K\beta(s+1)>\frac{K\beta r}{6}
$
and \(r/s\le48\).
It follows that
$
 \phi_{r,q,\beta}(x)p_r
 \le
 200\eta K\beta
 \left(
 1+\frac{r}{\log(2q)}
 \right).
$
Since \(x\le p_r/2\), we obtain
$
 |T|\ge
 \frac{p_r}{2}
 -200\eta K\beta
 \left(
 1+\frac{r}{\log(2q)}
 \right)-200\beta>\frac{p_r}{2}
 -\frac{K\beta}{10}
 \left(
 1+\frac{r}{\log(2q)}
 \right)+2d+2>p_{\lceil r/2\rceil}.
$
\hfill $\blacksquare$

If \(r=2\), then \(|T|\ge p_1=2d\), while
\(|S|\ge2d\). The two parts of \(J\) have orders at most \(2d\).
Since \(e_G(S,T)=0\), we may choose one part of \(J\) in \(S\) and the other in \(T\), obtaining a copy of \(J\) in \(\ol G\), a contradiction. Thus, when \(r=2\), no bad set exists, and Case 1 cannot happen.
%Case~2 gives the required subgraph. We may therefore assume that \(r\ge3\).
Let \(t\) be maximal such that \(p_t\le|T|\), and let
\(a\le r-t\) be maximal such that \(p_a\le x\).  Then
\(t\ge\lceil r/2\rceil\) and \(a\ge1\).  
\begin{claim}
\(a+t=r\).
\end{claim}

Suppose for contradiction that $a+t<r$. By the maximality of $a$ and $t$, we have $x<p_{a+1}$ and $|T|<p_{t+1}\le p_{r-a}$. Since $S$ is bad, it follows that $p_r<p_{a+1}+p_{r-a}+\phi_{r,q,\beta}(x)p_{a+1}+200\beta$.
Since \(p_a\le x\), we have \(s\ge a\). 
Moreover, the function
$
  z\longmapsto 1+\log(r/z)
$
is decreasing on \((0,\infty)\). Hence
$
  \phi_{r,q,\beta}(x)
  \le
  \eta\frac{1+\log(r/a)}{\log(2q)}.
$
Moreover,
$
  p_{a+1,q}(d,\beta)
  \le 4K\beta(a+1)
  \le 8K\beta a,
$
and hence
$
  \phi_{r,q,\beta}(x)p_{a+1,q}(d,\beta)
  \le
  8\eta K\beta
  \frac{a(1+\log(r/a))}{\log(2q)}
  \le
  8\eta K\beta
  \left(
    1+\frac{a(1+\log(r/a))}{\log(2q)}
  \right).
$
Moreover, $p_{a+1}\le4K\beta(a+1)\le8K\beta a$ and hence
$\phi_{r,q,\beta}(x)p_{a+1}\le
8\eta K\beta\bigl(1+a(1+\log(r/a))/\log(2q)\bigr)$.
On the other hand, 
$p_r-p_{r-a}-p_{a+1}\ge
(K\beta/10)\bigl(1+a(1+\log(r/a))/\log(2q)\bigr)-2d-2>
8\eta K\beta\bigl(1+a(1+\log(r/a))/\log(2q)\bigr)+200\beta$,
a contradiction. Hence $a+t=r$. \hfill $\blacksquare$

\medskip

Partition the parts of $J$ into sets of sizes $a$ and $t$, and let $J_S$ and $J_T$ be the corresponding part-subgraphs. Since $e_G(S,T)=0$, copies of $J_S$ in $\ol{G[S]}$ and $J_T$ in $\ol{G[T]}$ would together form a copy of $J$ in $\ol G$. Hence either \(\ol{G[S]}\) does not contain \(J_S\), or \(\ol{G[T]}\) does not contain \(J_T\).
Suppose first that \(a=1\). Since \(p_1=2d\le|S|\) and the unique part of \(J_S\) has order at most \(2d\), the graph \(\ol{G[S]}\) contains a copy of \(J_S\). Hence \(\ol{G[T]}\) is \(J_T\)-free. Moreover, \(t=r-1\), so \(2\le t<r\), and \(|T|\ge p_t\). The induction hypothesis applied to \(G[T]\) and \(J_T\) gives the required subgraph. 
Suppose that \(a\ge2\). Then \(2\le a,t<r\), and
\(|S|\ge p_a\) and \(|T|\ge p_t\). If \(\ol{G[S]}\) does not contain \(J_S\), apply the induction hypothesis to \(G[S]\) and \(J_S\).
Otherwise \(\ol{G[T]}\) does not contain \(J_T\), and the induction hypothesis applies to \(G[T]\) and \(J_T\). In both cases, \(d,\beta\), and \(q\) remain unchanged.

\medskip
\noindent \textbf{Case 2.}
No bad set exists.  
\medskip

If there is a nonempty set \(X\subseteq V(G)\) with \(|X|\le4d\) and
\(|N_G(X)|<20|X|\), choose one of maximum size; otherwise let
\(X:=\varnothing\).  
We have \(4d< p_r/2\). Thus $|X|<2d$, since otherwise $2d\le |X|\le p_r/2$ and $|N_G(X)|<20|X|\le200\beta\le
\phi_{r,q,\beta}(|X|)|X|+200\beta$, so $X$ would be bad.
Let \(F:=G-X\). Since \(r\ge2\),
$
 |F|\ge p_r-2d\ge K\beta>2\beta,
$
and hence \(\beta<|F|/2\).
Let \(Y\subseteq V(F)\) be nonempty with \(|Y|\le|F|/2\). 
Suppose that \(0<|Y|\le2d\) and \(|N_F(Y)|<20|Y|\). Then
\(|X\cup Y|<4d\) and
$
 |N_G(X\cup Y)|
 \le |N_G(X)|+|N_F(Y)|
 <20|X\cup Y|,
$
a contradiction to the choice of \(X\).  
If $2d\le |Y|\le\beta$, then $\beta< p_r/2$, so $|N_F(Y)|\ge |N_G(Y)|-|X|>200\beta-2d\ge20|Y|$.
Finally, if \(\beta\le|Y|\le|F|/2\), then
$
 |N_F(Y)|>
 \phi_{r,q,\beta}(|Y|)|Y|+200\beta-2d
 \ge
 \phi_{r,q,\beta}(|Y|)|Y|+10\beta.
$
Thus \(F\) multi-\((d,r,q,\beta,20)\)-expands.  Since
\(|X|<2d\), we also have \(p_r-2d\le|F|\le p_r\),
as required.
\end{proof}

\subsection{Expansion properties}
The expansion properties established above give us several tools needed to construct protected switchers.
Define
\begin{equation}\label{eq:path-bound}
 E(r,q,\beta):=
 \left\lceil
 10^6\bigl(
  \log(2\beta)+\log(2q)\log\log(4r)
 \bigr)
 \right\rceil.
\end{equation}
The first lemma allows us to join two
prescribed sets, while avoiding a smaller set, by a path which is not too long.
\begin{lemma}\label{lem:short-path}
Let $\beta\ge qd$. Suppose that $F$ multi-$(d,r,q,\beta,16)$-expands into $W$ and that $|F|\le p_{r,q}(d,2\beta)$. Let $A,B,Z$ be pairwise disjoint subsets of $V(F)$ with $A,B\ne\varnothing$, and suppose that
$
 |Z\cap W|\le14\min\{|A|,|B|\},
$
and
$
|Z|\le5\beta.
$
Then $F-Z$ contains an $A$--$B$ path of length at most
$E(r,q,\beta)$.
\end{lemma}

\begin{proof}
Set \(A_0:=A\) and
\(A_{i+1}:=(A_i\cup N_F(A_i))\setminus Z\). For \(|A_i|\le\beta\), we have
$
 |A_{i+1}|
 \ge |A_i|+16|A_i|-|Z\cap W|
 \ge 3|A_i|.
$
Thus,
\(|A_a|\ge\beta\), where \(a:=\lceil\log_3\beta\rceil\).
Now suppose that \(\beta\le|A_i|\le|F|/2\). Since \(|Z|\le5\beta\), we have
\begin{equation}\label{eq:BFS-medium}
 |A_{i+1}|
 \ge
 \bigl(1+\phi_{r,q,\beta}(|A_i|)\bigr)|A_i|.
\end{equation}
Let \(s\) be the integer used in the definition of
\(\phi_{r,q,\beta}(|A_i|)\). We first consider the case when \(s=1\).
Then \(\phi_{r,q,\beta}(|A_i|)\ge
\varepsilon_0:=\eta(1+\log 2)/\log(2q)\). If \(r\ge3\), the maximality
of \(s\) implies that \(|A_i|<p_{2,q}(d,\beta)\). When \(r=2\), it
suffices to show that \(|A_a|>|F|/2\), and
\(|F|/2\le p_{2,q}(d,2\beta)/2\le p_{2,q}(d,\beta)\). Moreover,
\(p_{2,q}(d,\beta)\le2K\beta\), since \(\beta\ge qd\). Thus, either
\(|A_j|>|F|/2\) for some \(j\le i+b\), or
\(|A_{i+b}|\ge p_{2,q}(d,\beta)\), where
\(b:=\lceil\log_{1+\varepsilon_0}(2K)\rceil\). Since
\(\log(1+x)\ge x/2\) for \(0\le x\le1\) and \(\beta\ge q\), we have \(a+b\le150000\log(2\beta)\).

We may now assume that \(s\ge2\).  For each \(j\ge0\), consider those values of \(i\) for which
$
 \frac{r}{2^{j+1}}<s\le\frac{r}{2^j}.
$
Then \(1+\log(r/s)\ge(j+1)/2\), and consequently
\begin{equation}\label{eq:range-expansion}
 \phi_{r,q,\beta}(|A_i|)
\ge
\frac{\eta}{2}\frac{j+1}{\log(2q)}.
\end{equation}

Fix a nonempty dyadic range \(\frac{r}{2^{j+1}}<s\le
\frac{r}{2^j}\). Since \(s\ge2\),  \(\lfloor r/2^{j+1}\rfloor+1\le s\le \min\{r-1,\lfloor r/2^j\rfloor\}\). Also,
\(\frac12K\beta w\le p_{w,q}(d,\beta)\le4K\beta w\) for every
\(2\le w\le q\). Since
\(\min\{r,\lfloor r/2^j\rfloor+1\}\le
2(\lfloor r/2^{j+1}\rfloor+1)\), it follows that
\(p_{\min\{r,\lfloor r/2^j\rfloor+1\},q}(d,\beta)
\le16p_{\lfloor r/2^{j+1}\rfloor+1,q}(d,\beta)\).
Whenever \(s\) belongs to this range, we have
\(p_{\lfloor r/2^{j+1}\rfloor+1,q}(d,\beta)
\le p_{s,q}(d,\beta)\le |A_i|\). If \(s\le r-2\), then the maximality
of \(s\) gives
\(|A_i|<p_{s+1,q}(d,\beta)\le
p_{\min\{r,\lfloor r/2^j\rfloor+1\},q}(d,\beta)\). If \(s=r-1\),
then, since the argument is needed only while \(|A_i|\le |F|/2\), we
have \(|A_i|\le |F|/2\le p_{r,q}(d,2\beta)/2
\le p_{r,q}(d,\beta)
=p_{\min\{r,\lfloor r/2^j\rfloor+1\},q}(d,\beta)\). Here the
penultimate inequality follows from
\(\lceil2x\rceil\le2\lceil x\rceil\). Hence, throughout this dyadic
range,
\(p_{\lfloor r/2^{j+1}\rfloor+1,q}(d,\beta)\le |A_i|
\le16p_{\lfloor r/2^{j+1}\rfloor+1,q}(d,\beta)\).
When \(r=2\), the definition forces \(s=1\) at every stage. Hence the
preceding argument applies throughout and, within \(b\) further rounds,
either \(|A_j|>|F|/2\) for some \(j\), or
\(|A_{i+b}|\ge p_{2,q}(d,\beta)\ge |F|/2\). In the latter case, one
further round, if necessary, makes the set larger than \(|F|/2\).

For the \(j\)th dyadic range, set
\(\varepsilon_j:=\frac{\eta}{2}(j+1)/\log(2q)\).
By \eqref{eq:BFS-medium} and \eqref{eq:range-expansion}, the reachable
set increases in each round by a factor of at least
\(1+\varepsilon_j\). Since its size needs to increase by a factor of
at most \(16\) within this range, the number \(N_j\) of corresponding
rounds satisfies
$
N_j\le\left\lceil\frac{\log16}{\log(1+\varepsilon_j)}\right\rceil.
$
As \(\varepsilon_j<1\), we have
\(\log(1+\varepsilon_j)\ge\varepsilon_j/2\). Thus, using \(\eta=10^{-4}\),
\(N_j\le\frac{4\log16}{\eta}\frac{\log(2q)}{j+1}+1
\le\frac{111000\log(2q)}{j+1}+1\).
Set \(J:=\lceil\log_2r\rceil\). Since
$
\sum_{j=0}^{J}\frac1{j+1}\le4\log\log(4r)
$
and
$
J+1\le4\log(2q)\log\log(4r),
$
the total number of rounds corresponding to all dyadic ranges is at
most
$
\sum_{j=0}^{J}N_j
\le450000\log(2q)\log\log(4r).
$

The initial small-set stage and the case \(s=1\) require at most
\(150000\log(2\beta)\) rounds. Hence each of the two sequences reaches
more than half of the vertices of \(F\) within
$
150000\log(2\beta)
+450000\log(2q)\log\log(4r)
\le\frac{E(r,q,\beta)}2
$
rounds. 
Define \(B_0:=B\) and
\(B_{i+1}:=(B_i\cup N_F(B_i))\setminus Z\). The same argument applies
to the sequence \(B_i\). Hence there are \(a',b'\le E(r,q,\beta)/2\)
such that \(|A_{a'}|>|F|/2\) and \(|B_{b'}|>|F|/2\). These two sets
intersect. Since \(A_i\) and \(B_i\) are the sets reachable from
\(A\) and \(B\), respectively, by paths of length at most \(i\) in
\(F-Z\), their intersection gives an \(A\)--\(B\) path in \(F-Z\) of
length at most \(a'+b'\le E(r,q,\beta)\).
\end{proof}

The next lemma is a variable-rate analogue of Lemma~4.5 in~\cite{HHKL}.
It gives a short odd cycle while preserving expansion into the remaining vertices.
\begin{lemma}\label{lem:compact-odd}
Let $\beta\ge qd$. Suppose that $F$ is non-bipartite, that $F$
multi-$(d,r,q,\beta,20)$-expands into $W$ and $|F|\le p_{r,q}(d,2\beta)$. Then $F$ contains an odd cycle $C$ with
$
 |C|\le2E(r,q,\beta)+1
$
such that $F$ multi-$(d,r,q,\beta,16)$-expands into $W\setminus V(C)$.
\end{lemma}

\begin{proof}
Choose an odd cycle \(C\subseteq F\) of minimum length in \(F\). We first note that
\(\text{dist}_F(x,y)=\text{dist}_C(x,y)\) for every \(x,y\in V(C)\), where \(\text{dist}_G(x,y)\) denotes the distance between \(x\) and \(y\) in \(G\).
Applying Lemma~\ref{lem:short-path} with \(Z=\varnothing\), we see that any two vertices of \(F\) are connected by a path of length at most \(E(r,q,\beta)\). Choose two almost-antipodal vertices \(x,y\in V(C)\). It follows that
$
 \frac{|C|-1}{2}
 =\text{dist}_C(x,y)
 =\text{dist}_F(x,y)
 \le E(r,q,\beta),
$
and hence \(|C|\le2E(r,q,\beta)+1\).

We next show that every vertex of \(F\) has at most three neighbors
on \(C\). This is immediate when \(C\) is a triangle. Suppose that
\(|C|\ge5\) and that some vertex \(v\) has at least three neighbors
on \(C\). If two of these neighbors are consecutive on \(C\), then
they form a triangle with \(v\), a contradiction to the choice of \(C\). Otherwise, two of them have distance at least three on \(C\), whereas they are connected by a path of length two through \(v\), a contradiction to \(\text{dist}_F(x,y)=\text{dist}_C(x,y)\).
Consequently, for every \(S\subseteq V(F)\), we have
$
 |N_F(S)\cap V(C)|\le3|S|.
$
If \(|S|\le\beta\), then 
$
 |N_F(S)\cap(W\setminus V(C))|
 \ge |N_F(S)\cap W|-|N_F(S)\cap V(C)|
 \ge20|S|-3|S|
 \ge16|S|.
$
For \(\beta\le|S|\le|F|/2\), the second expansion condition does not
depend on \(W\), and so it remains unchanged. Thus \(F\)
multi-\((d,r,q,\beta,16)\)-expands into \(W\setminus V(C)\).
\end{proof}

%{\color{red} Expansion in old references.}
\begin{definition}\label{def:rooted-set}
Let $F$ be a graph, let $x\in V(F)$, and let $R\ge1$ be an integer. A set
$A\subseteq V(F)$ is an \textit{$(x,R)$-rooted set} if $x\in A$, $|A|=R$, and every vertex of $A$ is joined to $x$ by a path in $F[A]$ of length at most $\lceil\log_8R\rceil$.
\end{definition}

The following lemma establishes the existence of pairwise disjoint rooted sets for up to four roots.
\begin{lemma}\label{lem:multi-root}
Suppose that $F$ multi-$(d,r,q,\beta,16)$-expands into $W$. Let
$1\le t\le4$, let $x_1,\ldots,x_t$ be distinct vertices, and let
$R\le\beta/t$. Then there are pairwise disjoint sets
$A_1,\ldots,A_t\subseteq W\cup\{x_1,\ldots,x_t\}$ such that $A_i$ is an $(x_i,R)$-rooted set.
\end{lemma}

\begin{proof}
If \(R=1\), take \(A_i=\{x_i\}\) for every \(i\).
Hence assume that \(R\ge2\).
We construct the sets $A_1, \dots, A_t$ iteratively. For each $i \in [t]$, set $A_i^{(0)} = \{x_i\}$. Suppose inductively that after step $j \ge 0$, the sets $A_1^{(j)}, \dots, A_t^{(j)}$ are pairwise disjoint, each of order $u = 8^j < R$, and that every vertex in $A_i^{(j)}$ has distance at most $j$ from $x_i$ in $F[A_i^{(j)}]$.

Set $U = \bigcup_{i=1}^t A_i^{(j)}$. Since $|U| = tu < tR \le \beta$, for every index subset $I \subseteq [t]$, the union $\bigcup_{i \in I} A_i^{(j)}$ has order $|I|u < \beta$. Thus, $|N_F(\bigcup_{i \in I} A_i^{(j)}) \cap W| \ge 16 |I| u$ and  $|\bigcup_{i \in I} ((N_F(A_i^{(j)}) \cap W) \setminus U)| \ge 16 |I| u - |U| = 16 |I| u - t u \ge 12 |I| u > 7 |I| u$, where the third inequality follows from $t \le 4 \le 4|I|$.

Let \(\Gamma\) be the bipartite graph with vertex classes
\([t]\times[7u]\) and \(W\setminus U\), where
\(N_\Gamma(i,h)=(N_F(A_i^{(j)})\cap W)\setminus U\) for every
\((i,h)\in[t]\times[7u]\). For any
\(\mathcal X\subseteq[t]\times[7u]\), let
\(I:=\{i\in[t]:\mathcal X\cap(\{i\}\times[7u])\ne\varnothing\}\).
Then
$
 N_\Gamma(\mathcal X)
 =
 \bigcup_{i\in I}\bigl((N_F(A_i^{(j)})\cap W)\setminus U\bigr),
$
and hence \(|N_\Gamma(\mathcal X)|\ge7u|I|\ge|\mathcal X|\). By Hall's theorem, there exists a matching in \(\Gamma\) covering \([t]\times[7u]\). For each \(i\in[t]\), let \(N_i^{(j)}\) be the set of vertices matched to \((i,1),\ldots,(i,7u)\). Then the sets \(N_1^{(j)},\ldots,N_t^{(j)}\)
are pairwise disjoint,
\(N_i^{(j)}\subseteq(N_F(A_i^{(j)})\cap W)\setminus U\), and
\(|N_i^{(j)}|=7u\) for every \(i\in[t]\).

Set \(A_i^{(j+1)}:=A_i^{(j)}\cup N_i^{(j)}\) for each \(i\in[t]\).
These sets are pairwise disjoint and each has order \(8u\). Moreover, every vertex of \(A_i^{(j+1)}\) is joined to \(x_i\) by a path in \(F[A_i^{(j+1)}]\) of length at most \(j+1\).
Continue in this way until \(u<R\le8u\). For each \(i\in[t]\), choose \(M_i^{(j)}\subseteq N_i^{(j)}\) with \(|M_i^{(j)}|=R-u\), and set \(A_i:=A_i^{(j)}\cup M_i^{(j)}\). Then the sets
\(A_1,\ldots,A_t\) are pairwise disjoint and each has order \(R\).
The construction takes \(\lceil\log_8R\rceil\) steps, and every vertex of \(A_i\) is joined to \(x_i\) by a path in \(F[A_i]\) of length at most \(\lceil\log_8R\rceil\). Hence \(A_i\) is an
\((x_i,R)\)-rooted set for every \(i\in[t]\).
\end{proof}

The next lemma iteratively applies Lemma~\ref{lem:short-path} to build a path of length within a given range.

\begin{lemma}\label{lem:coarse-path}
Let \(\mathcal{J}\) be a complete \(q\)-partite graph with each part of size in
\([d,2d]\), let \(J\subseteq\mathcal{J}\) have \(r\ge2\) parts, and let
\(F_0\) satisfy
\begin{enumerate}[label=\textup{(\roman*)}]
\item \(\ol{F_0}\) is \(J\)-free and
$
 p_{r,q}(d,2\beta)-2d\le |F_0|\le p_{r,q}(d,2\beta);
$
\item \(F_0\) multi-\((d,r,q,\beta,16)\)-expands into \(W\);
\item \(\beta\ge qd\) and
$
 4E(r,q,\beta)\le R:=\lfloor\beta/1000\rfloor.
$
\end{enumerate}
Let \(A,B\) be disjoint \((x,R)\)-rooted set and \((y,R)\)-rooted set.  Let
\(Z\subseteq V(F_0)\setminus(A\cup B)\), and let \(T\ge0\) be an integer
satisfying
\begin{align}
 |Z|+T+4E(r,q,\beta)+2R&\le5\beta,
 \label{eq:coarse-size}\\
 |Z\cap W|+T+4E(r,q,\beta)&\le12R.
 \label{eq:coarse-W}
\end{align}
Then \(F_0-Z\) contains an \(x\)--\(y\) path \(P\) with
$
 T\le\ell(P)\le T+4E(r,q,\beta).
$
\end{lemma}

\begin{proof}
Write \(E:=E(r,q,\beta)\). Since \(R\le\beta\), we have
\(2\lceil\log_8R\rceil+2\le
2\lceil\log_8(\beta+1)\rceil+2\le7\log(2\beta)\le E\).
We construct a path \(Q\) starting at \(x\), together with an
\((z,R)\)-rooted set \(D\), where \(z\) is the other endvertex of
\(Q\). At every stage,
\(Q\cap D=\{z\}\), \(D\cap(Z\cup B)=\varnothing\), and
\(Q\cap(Z\cup B)=\varnothing\). Initially, \(Q\) is the trivial path
at \(x\), and \(D=A\).

Suppose that \(\ell(Q)<T\), and set
\(U:=Z\cup V(Q)\cup D\cup B\). Since \(\ell(Q)\le T-1\), we have
$
|U|\le|Z|+\ell(Q)+2R<|Z|+T+2R\le5\beta
$
by \eqref{eq:coarse-size}. Hence
$
|F_0-U|\ge p_{r,q}(d,2\beta)-2d-5\beta
\ge p_{r,q}(d,\beta)+15\beta\ge p_{r,q}(d,\beta).
$
Moreover, \(\ol{F_0-U}\) does not contain \(J\). By
Lemma~\ref{lem:expanding-subgraph}, there are \(2\le r'\le r\) and an
induced subgraph \(F'\subseteq F_0-U\) such that \(F'\) multi-\((d,r',q,\beta,20)\)-expands. In particular,
\(\delta(F')\ge20\), so choose an edge \(uv\in E(F')\). Since
\(R\le\beta/2\), Lemma~\ref{lem:multi-root}, applied in \(F'\) with
roots \(u,v\), gives disjoint \((u,R)\)-rooted set \(D_u\) and \((v,R)\)-rooted set
\(D_v\).
Apply Lemma~\ref{lem:short-path} in \(F_0\) from \(D\) to \(D_u\),
avoiding \(C:=Z\cup(V(Q)\setminus\{z\})\cup B\cup D_v\). Also, \(C\) is disjoint from \(D\cup D_u\).
Moreover,
$
|C|\le|Z|+\ell(Q)+2R<|Z|+T+2R\le5\beta
$
and
$
|C\cap W|\le|Z\cap W|+\ell(Q)+2R
<|Z\cap W|+T+2R\le14R,
$
where the last inequality follows from \eqref{eq:coarse-W}.

Orient the resulting path \(P_0\) from \(D\) to \(D_u\). Let \(a\) be
its last vertex in \(D\), let \(b\) be the first vertex of \(D_u\)
after \(a\), and let \(P_0'\) be the \(a\)--\(b\) subpath of \(P_0\).
Then \(V(P_0')\cap D=\{a\}\) and
\(V(P_0')\cap D_u=\{b\}\).
Choose a \(z\)--\(a\) path \(Q_D\) in \(F_0[D]\) and a \(b\)--\(u\)
path \(Q_u\) in \(F_0[D_u]\), each of length at most
\(\lceil\log_8R\rceil\). Then
\(Q\cup Q_D\cup P_0'\cup Q_u\cup uv\) is an \(x\)--\(v\) path, and
the number of edges added to \(Q\) is at most
\(E+2\lceil\log_8R\rceil+1\le2E\). Replace \(Q\) by this
path and \(D\) by \(D_v\).

Continue until \(\ell(Q)\ge T\) for the first time. Since the final
extension adds at most \(2E\) edges, we have
\(T\le\ell(Q)\le T+2E\).
Apply Lemma~\ref{lem:short-path} from \(D\) to \(B\), avoiding
\(C':=Z\cup(V(Q)\setminus\{z\})\). By the invariants, \(C'\) is
disjoint from \(D\cup B\). Moreover,
$
|C'|\le|Z|+\ell(Q)\le|Z|+T+2E\le5\beta
$
and
$
|C'\cap W|\le|Z\cap W|+\ell(Q)
\le|Z\cap W|+T+2E\le12R\le14R.
$
Hence there is a \(D\)--\(B\) path \(P_1\) of length at most \(E\)
avoiding \(C'\).
Orient \(P_1\) from \(D\) to \(B\). Let \(a_1\) be its last vertex in
\(D\), let \(b_1\) be the first vertex of \(B\) after \(a_1\), and let
\(P_1'\) be the \(a_1\)--\(b_1\) subpath of \(P_1\). Choose a
\(z\)--\(a_1\) path in \(F_0[D]\) and a \(b_1\)--\(y\) path in
\(F_0[B]\), each of length at most \(\lceil\log_8R\rceil\). Together
with \(Q\) and \(P_1'\), these paths form an \(x\)--\(y\) path \(P\)
in \(F_0-Z\), see Figure \ref{fig:coarse-path}. The final extension has length at most
\(E+2\lceil\log_8R\rceil\le2E\). Therefore
$
T\le\ell(P)\le T+4E=T+4E(r,q,\beta),
$
as required.
\end{proof}

\begin{figure}[t]
\centering
\resizebox{.96\linewidth}{!}{%
\begin{tikzpicture}[
  x=1cm,
  y=1cm,
  every node/.style={font=\scriptsize},
  vertex/.style={circle,fill=black,inner sep=1.1pt},
  rooted/.style={draw=black,rounded corners=3pt,line width=.7pt},
  mainpath/.style={draw=black,line width=.9pt},
  auxpath/.style={draw=black,dashed,line width=.8pt}
]

% initial path Q
\node[vertex,label=above:$x$] (x) at (0.4,1.55) {};
\draw[mainpath] (x)--(1.3,1.55)--(2.1,1.15)--(2.9,1.95)--(3.7,1.55);
\node at (2.1,2.2) {$Q$};

% first active rooted set D
\draw[rooted] (3.4,0.95) rectangle (4.9,2.15);
\node at (4.15,2.4) {$D$};
\node[vertex,label=above:$z$] (z) at (3.85,1.55) {};
\node[vertex,label=below:$a$] (a) at (4.55,1.28) {};
\draw[mainpath] (z)--(a);

% D_u
\draw[rooted] (5.9,0.95) rectangle (7.4,2.15);
\node at (6.65,2.4) {$D_u$};
\node[vertex,label=below:$b$] (b) at (6.15,1.28) {};
\node[vertex,label=above:$u$] (u) at (6.85,1.55) {};
\draw[mainpath] (b)--(u);

% short path P_0'
\draw[mainpath] (a)--(5.1,0.68)--(5.7,0.96)--(b);
\node at (5.4,0.33) {$P_0'$};

% edge uv and D_v
\node[vertex,label=above:$v$] (v) at (8.0,1.55) {};
\draw[mainpath] (u)--(v);
\draw[rooted] (7.65,0.95) rectangle (9.15,2.15);
\node at (8.40,2.4) {$D_v$};

% continuation
\draw[auxpath] (9.35,1.55)--(10.25,1.55);
\node at (10.7,1.55) {$\cdots$};
\draw[auxpath] (11.15,1.55)--(12.05,1.55);

% final active rooted set D
\draw[rooted] (12.3,0.95) rectangle (13.8,2.15);
\node at (13.05,2.4) {$D$};
\node[vertex,label=above:$z$] (z2) at (12.75,1.55) {};
\node[vertex,label=below:$a_1$] (a1) at (13.45,1.28) {};
\draw[mainpath] (z2)--(a1);

% B
\draw[rooted] (14.8,0.95) rectangle (16.3,2.15);
\node at (15.55,2.4) {$B$};
\node[vertex,label=below:$b_1$] (b1) at (15.05,1.28) {};
\node[vertex,label=above:$y$] (y) at (15.75,1.55) {};
\draw[mainpath] (b1)--(y);

% final short path P_1'
\draw[mainpath] (a1)--(13.95,0.68)--(14.55,0.96)--(b1);
\node at (14.25,0.33) {$P_1'$};

\end{tikzpicture}%
}
\caption{The desired path $P$. 
Starting from the current path \(Q\) ending at \(z\in D\), a short path \(P_0'\) joins \(D\) to a rooted set \(D_u\), and the edge \(uv\) moves the endpoint to the rooted set \(D_v\). Repeating this step produces a path of length at least \(T\). A final short path \(P_1'\) then joins the last rooted set \(D\) to \(B\), yielding the required \(x\)--\(y\) path.}
\label{fig:coarse-path}
\end{figure}
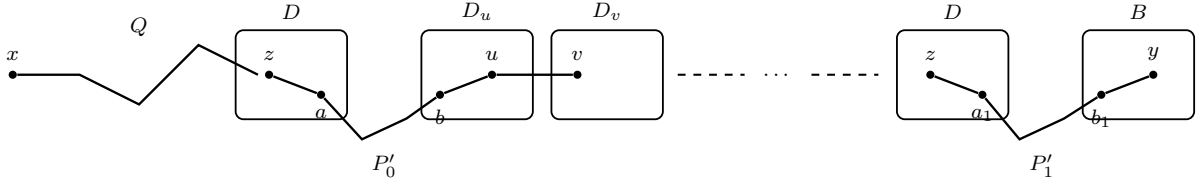

\section{Switchers}
To construct cycles covering a range of lengths, we shall need the following switching structure.

\begin{definition}\label{def:switch}
A \emph{switch} is a cycle $C$ with \textit{pivotal} vertices $v,w$.  Write $C^+$ and $C^-$ for its two $v$--$w$ paths, labeled so that
$\ell(C^+)>\ell(C^-)$.  The \emph{weight} of the switch is
$\delta(C)=\ell(C^+)-\ell(C^-)$.
\end{definition}
The following switcher structure is a generalization of an adjuster in \cite{HHKL}.
\begin{definition}\label{def:switcher}
A \emph{switcher} $F$ consists of pairwise vertex-disjoint switches $C_1,\ldots,C_r$, with pivotal vertices $v_i,w_i$, together with paths $P_1,\ldots,P_r$, where $P_i$ joins $w_i$ to $v_{i+1}$ and subscripts are taken cyclically. The paths are internally disjoint from one another and from the switches. A \emph{route} of \(F\) is the cycle obtained by using every \(P_i\) and,
for each \(i\), one of the two \(v_i\)--\(w_i\) paths in \(C_i\). The \emph{length} $\ell(F)$ of \(F\) is the length of the route using every long side $C_i^+$. A cycle is a switcher with no switches.

Orient the longest route according to the cyclic order of the switches. A \emph{block} is a nonempty consecutive set of switches. Its block subgraph is the union of these switches and the paths joining consecutive switches in the block. The \emph{order} of the block is the number of vertices in its block subgraph. Its \emph{support} is the corresponding subpath of the longest route, from the first pivotal vertex of its first switch to the second pivotal vertex of its last switch, see Figure \ref{fig:switcher-terminology}.

A block is \emph{complete} if the subset sums of its switch weights contain every integer between zero and the sum of its weights.

A \emph{protected switcher} is a switcher together with a partition of its switches into consecutive complete blocks whose supports are pairwise edge-disjoint. These blocks are called its \emph{protected blocks}. For a protected block \(B\), define
$
 \Omega(B):=\sum_{C\in B}\delta(C),
$
and define the \emph{flexibility} of \(F\) by
\[
 \Omega(F):=\sum_C\delta(C),
\]
where the sum is over all switches of \(F\). If \(F\) has no switches, the maximum protected-block order and flexibility are both taken to be zero.
\end{definition}

\begin{figure}[t]\label{f1}
\centering

\tikzset{
  every node/.style={font=\scriptsize},
  vertex/.style={circle,fill=black,inner sep=1.25pt},
  edge/.style={draw=black,line width=.45pt},
  route/.style={draw=blue!65!black,line width=.9pt},
  support/.style={draw=red!75!black,line width=1pt},
  blockbox/.style={draw=black!55,dashed,rounded corners=2pt,line width=.45pt},
  bleft/.style={draw=blue!55!black,fill=blue!5,rounded corners=3pt,line width=.45pt},
  bright/.style={draw=green!45!black,fill=green!5,rounded corners=3pt,line width=.45pt},
  pfive/.style={
    route,
    postaction={
      decorate,
      decoration={
        markings,
        mark=at position 0.84 with {\arrow{Stealth[length=1.5mm,width=1.2mm]}}
      }
    }
  }
}

%-------------------- (a) --------------------%
\begin{minipage}{\linewidth}
\centering
\begin{tikzpicture}[x=.92cm,y=.92cm,>=Stealth]
  \path[use as bounding box] (-.9,-2.05) rectangle (12.25,1.95);

  \node[font=\small] at (5.65,1.62) {switcher $F$};

  \foreach \i/\x in {1/0,2/2.35,3/4.70,4/7.05,5/9.40}{
    \coordinate (v\i) at (\x,0);
    \coordinate (w\i) at ({\x+1.45},0);
  }

  % switches
  \draw[route]   (v1) .. controls +(0,.72) and +(0,.72) .. (w1);
  \draw[edge]    (v1) .. controls +(0,-.62) and +(0,-.62) .. (w1);

  \draw[support] (v2) .. controls +(0,.72) and +(0,.72) .. (w2);
  \draw[edge]    (v2) .. controls +(0,-.62) and +(0,-.62) .. (w2);

  \draw[support] (v3) .. controls +(0,.72) and +(0,.72) .. (w3);
  \draw[edge]    (v3) .. controls +(0,-.62) and +(0,-.62) .. (w3);

  \draw[route]   (v4) .. controls +(0,.72) and +(0,.72) .. (w4);
  \draw[edge]    (v4) .. controls +(0,-.62) and +(0,-.62) .. (w4);

  \draw[route]   (v5) .. controls +(0,.72) and +(0,.72) .. (w5);
  \draw[edge]    (v5) .. controls +(0,-.62) and +(0,-.62) .. (w5);

  % joining paths
  \draw[route]   (w1)--(v2);
  \draw[support,-{Stealth[length=1.5mm,width=1.2mm]}] (w2)--(v3);
  \draw[route]   (w3)--(v4);
  \draw[route]   (w4)--(v5);

  % closing path P5
  \draw[pfive]
    (w5) .. controls +(0,-.55) and +(.55,0) .. (11.55,-1.75)
         -- (-.55,-1.75)
         .. controls +(-.55,0) and +(0,-.55) .. (v1);

  % block subgraph
  \draw[blockbox] (2.15,-.88) rectangle (6.35,.98);
  \node[fill=white,inner sep=1pt] at (4.25,1.00)
    {block subgraph of $B=\{C_2,C_3\}$};

  % support brace
  \draw[support,decorate,decoration={brace,mirror,amplitude=4pt}]
    (2.35,-1.15)--(6.15,-1.15)
    node[midway,yshift=-10pt,text=red!75!black] {$\operatorname{supp}(B)$};

  % labels
  \node at (.72,1.05) {switch $C_1$};

  \foreach \i in {1,...,5}{
    \node at ($(v\i)!.5!(w\i)+(0,.50)$) {$C_{\i}^{+}$};
    \node at ($(v\i)!.5!(w\i)+(0,-.48)$) {$C_{\i}^{-}$};
    \node[below=1pt] at (v\i) {$v_{\i}$};
    \node[below=1pt] at (w\i) {$w_{\i}$};
  }

  \node[above=1pt] at ($(w1)!.5!(v2)$) {$P_1$};
  \node[above=1pt] at ($(w2)!.5!(v3)$) {$P_2$};
  \node[above=1pt] at ($(w3)!.5!(v4)$) {$P_3$};
  \node[above=1pt] at ($(w4)!.5!(v5)$) {$P_4$};
  \node at (-.45,-1.43) {$P_5$};

  \foreach \i in {1,...,5}{
    \node[vertex] at (v\i) {};
    \node[vertex] at (w\i) {};
  }
\end{tikzpicture}

\smallskip
{\small (a)}
\end{minipage}

\vspace{0.8em}

%-------------------- (b) --------------------%
\begin{minipage}{\linewidth}
\centering
\begin{tikzpicture}[x=.92cm,y=.92cm,>=Stealth]
  \path[use as bounding box] (-.9,-2.05) rectangle (12.25,1.15);

  \foreach \i/\x in {1/0,2/2.35,3/4.70,4/7.05,5/9.40}{
    \coordinate (V\i) at (\x,0);
    \coordinate (W\i) at ({\x+1.45},0);
  }

  \begin{scope}[on background layer]
    \draw[bleft]  (-.30,-.83) rectangle (6.40,.81);
    \draw[bright] (6.82,-.83) rectangle (11.15,.81);
  \end{scope}

  \foreach \i in {1,...,5}{
    \draw[edge] (V\i) .. controls +(0,.58) and +(0,.58) .. (W\i);
    \draw[edge] (V\i) .. controls +(0,-.58) and +(0,-.58) .. (W\i);
    \node[vertex] at (V\i) {};
    \node[vertex] at (W\i) {};
    \node at ($(V\i)!.5!(W\i)+(0,.42)$) {$C_{\i}$};
  }

  \draw[edge] (W1)--(V2);
  \draw[edge] (W2)--(V3);
  \draw[edge] (W3)--(V4);
  \draw[edge] (W4)--(V5);

  \node[above=1pt] at ($(W1)!.5!(V2)$) {$P_1$};
  \node[above=1pt] at ($(W2)!.5!(V3)$) {$P_2$};
  \node[above=1pt] at ($(W3)!.5!(V4)$) {$P_3$};
  \node[above=1pt] at ($(W4)!.5!(V5)$) {$P_4$};

  \draw[blue!60!black,decorate,decoration={brace,mirror,amplitude=4pt}]
    (-.20,-1.05)--(6.30,-1.05)
    node[midway,yshift=-10pt,text=blue!60!black]
    {protected block $B_1=\{C_1,C_2,C_3\}$};

  \node[text=blue!60!black] at (3.05,-1.78)
    {weights $1,2,4$, \quad $\Omega(B_1)=7$};

  \draw[green!45!black,decorate,decoration={brace,mirror,amplitude=4pt}]
    (6.92,-1.05)--(11.05,-1.05)
    node[midway,yshift=-10pt,text=green!45!black]
    {protected block $B_2=\{C_4,C_5\}$};

  \node[text=green!45!black] at (8.98,-1.78)
    {weights $1,1$, \quad $\Omega(B_2)=2$};
\end{tikzpicture}

\smallskip
{\small (b)}
\end{minipage}

\caption{
(a) A switcher $F$, together with the block subgraph of the block $B=\{C_2,C_3\}$ and its support $\operatorname{supp}(B)$. The blue route is the longest route, and the red subpath indicates the support. 
(b) A protected switcher whose switches are partitioned into two protected blocks $B_1$ and $B_2$, together with their weights and flexibilities.}
\label{fig:switcher-terminology}
\end{figure}

If every switch is an odd cycle with almost-antipodal
vertices, then every switch has weight one. Subject to the bound on the cycle lengths in Definition~3.5 in \cite{HHKL}, this is the structure used in that paper.

The completeness of the protected blocks gives a full interval of route lengths.

\begin{lemma}\label{lem:interval-routes}
A protected switcher \(F\) contains a route of every integer length in
$
 [\ell(F)-\Omega(F),\ell(F)].
$
\end{lemma}

\begin{proof}
Let \(B_1,\ldots,B_s\) be the protected blocks. Since \(B_i\) is complete,
the switches in \(B_i\) can decrease the route length by any integer in
\([0,\Omega(B_i)]\). Hence the possible total decreases contain
$
 [0,\Omega(B_1)]+\cdots+[0,\Omega(B_s)]
 =
 \left[0,\sum_{i=1}^s\Omega(B_i)\right]
 =
 [0,\Omega(F)].
$
\end{proof}

Let \(\mathcal B(F)\) be the union of the protected block subgraphs.
A subpath \(Q\) of the longest route is \emph{free} if it is contained in every route and
$
 V(Q)\cap V(\mathcal B(F))=\varnothing.
$
A \emph{free family} is a collection \(\mathcal Q\) of pairwise
edge-disjoint free subpaths. Write
$
 W(\mathcal Q):=\sum_{Q\in\mathcal Q}\ell(Q),
$
and
$
 p(\mathcal Q):=|\mathcal Q|.
$
Since every replacement considered below uses only vertices of a free subpath, it cannot meet either side of a protected switch.

The following elementary observation allows us to shorten a free subpath without changing any switch.

\begin{lemma}\label{lem:shortcut}
Let $J$ be a complete $a$-partite graph of order $j$, where $a\ge2$, and
suppose that $\ol G$ is $J$-free. Then the following hold.
\begin{enumerate}[label=\textup{(\roman*)}]
\item Every induced path in $G$ has fewer than $j+a-1$ vertices.
\item Let \(F\) be a protected switcher in \(G\), and let \(Q\) be a
free subpath of \(F\). Suppose that
$
 \ell(Q)\ge j+a-2
$
and that every route of \(F\) has at least \(j+a\) edges. Then a subpath of \(Q\) can be replaced so that every route is shortened by the same
integer in
$
 [1,j+a-3],
$
while every switch remains unchanged.
\end{enumerate}
\end{lemma}

\begin{proof}
Write the part sizes of \(J\) as \(b_1,\ldots,b_a\). If \(P\) is an
induced path on \(j+a-1\) vertices, then along \(P\) we may choose
consecutive sets \(B_1,\ldots,B_a\) with \(|B_i|=b_i\), leaving one
unused vertex between consecutive sets. Since \(P\) is induced, there are no edges between distinct \(B_i\), so these sets span a copy of \(J\) in \(\ol G\), a contradiction. This proves~(i).

For~(ii), let \(d:=j+a-2\), and choose a subpath \(Q'\subseteq Q\) with \(d\) edges and endvertices \(x,y\). Let \(P'\) be a shortest
\(x\)--\(y\) path in \(G[V(Q')]\). By~(i),
\(1\le \ell(P')\le d-1\). In every route, the complementary
\(x\)--\(y\) path has at least two edges and is internally disjoint from \(P'\). Hence \(Q'\) may be replaced by \(P'\), shortening every route by \(d-\ell(P')\in[1,j+a-3]\). Since \(Q\) is free, no switch is changed.
\end{proof}

The next lemma shows that, given many short vertex-disjoint paths from a protected switcher to a disjoint cycle, two can be chosen to incorporate almost all of the cycle while losing only a small fraction of the switcher's length, flexibility, and free length.
\begin{lemma}\label{lem:protected-splice}
Let $F$ be a protected switcher of length $\ell$
and flexibility $\Omega$, and suppose that every protected block has order
at most $b\ge1$ and flexibility at most $\omega$.  Let $C$ be a cycle of
length $c$, disjoint from $F$.  Suppose that $q\ge400$ pairwise
vertex-disjoint paths join $F$ to $C$, each of length at most $t$.  Then two
of these paths give a protected switcher $F'$ satisfying
\begin{align}
 \ell+c+2t\geq \ell(F')&\ge\left(1-\frac{1000}q\right)\ell
             +0.99c-4b,\label{eq:protected-join-low}\\
 \Omega(F')&\ge\left(1-\frac{1000}q\right)\Omega
             -4\omega.\label{eq:protected-join-flex}
\end{align}
If $F$ has a free family $\mathcal Q$ with $W(\mathcal Q)=W$ and
$p(\mathcal Q)=p$, the paths may be chosen so that $F'$ has a free family
$\mathcal Q'$ satisfying
\begin{align}
 W(\mathcal Q')&\ge\left(1-\frac{1000}q\right)W+0.99c-1,
 \label{eq:protected-join-free}\\
 p(\mathcal Q')&\le p+3.\label{eq:protected-join-p}
\end{align}
Every protected block of \(F'\) is an unchanged protected block of
\(F\), and hence has order at most \(b\) and flexibility at most
\(\omega\).
\end{lemma}

\begin{proof}
Shorten each joining path \(P_i\) to the subpath from its last vertex in \(F\) to its first vertex in \(C\), and write \(x_i\in V(F)\) and \(y_i\in V(C)\) for its endvertices. Label internal vertices on long and short switch sides by \(+\) and \(-\), respectively; label vertices on connecting paths and pivotal vertices arbitrarily. By the pigeonhole principle, there is a set \(I_0\subseteq[q]\) with
\(|I_0|\ge q/2\) such that the vertices \(x_i\), \(i\in I_0\), have the same label.

Let \(R\) be the longest route of \(F\). If \(x_i\) lies internally on a short side, let \(x_i^*\) be the vertex on the corresponding long side at the same distance from the first pivotal vertex; otherwise set \(x_i^*=x_i\). Partition \(C\) into $100$ consecutive arcs of length at most \(c/100+1\), assigning each common endvertex to one of the two arcs. Some arc \(I_C\) contains at least \(q/200\) vertices \(y_i\) with \(i\in I_0\). Let \(\mathcal I\) be the corresponding set of indices and set \(m:=|\mathcal I|\). Then \(m\ge q/200\ge2\).

Let \(I_1,\ldots,I_m\) be the subpaths of \(R\) between consecutive
vertices \(x_i^*\), \(i\in\mathcal I\), in cyclic order. For
\(s\in[m]\), set $w_s:=\left|E(I_s)\cap\bigcup_{Q\in\mathcal Q}E(Q)\right|$ and $\omega_s:=\sum_{\operatorname{supp}(B)\subseteq I_s}\Omega(B)$,
where the second sum is over the protected blocks of \(F\). Since
the edge sets of \(I_1,\ldots,I_m\) partition \(E(R)\), we have
\(\sum_s\ell(I_s)=\ell\), \(\sum_sw_s=W\), and
\(\sum_s\omega_s\le\Omega\). Hence one of these paths, denoted by
\(I_F\), has endvertices \(x_i^*,x_j^*\) for distinct
\(i,j\in\mathcal I\) and satisfies $\ell(I_F)\le\frac{4\ell}{m}$, $w(I_F)\le\frac{4W}{m}$ and $\omega(I_F)\le\frac{4\Omega}{m}$.
Since \(m\ge q/200\),
\begin{equation}\label{w}
\ell(I_F)\le\frac{1000\ell}{q},\qquad
w(I_F)\le\frac{1000W}{q},\qquad
\omega(I_F)\le\frac{1000\Omega}{q}.
\end{equation}

Since \(x_i\) and \(x_j\) have the same label, there is a route
\(R_{ij}\) containing both vertices and using the long side in every
other switch. Let \(D_F\) be the \(x_i\)--\(x_j\) subpath of
\(R_{ij}\) corresponding to \(I_F\). Then
\(\ell(D_F)\le\ell(I_F)+2b\) and
\(\ell(R_{ij})\ge\ell-2b\). Let \(D_C\) be the \(y_i\)--\(y_j\)
subpath of \(I_C\); then \(\ell(D_C)\le c/100+1\).

Let \(\overline D_F\) and \(\overline D_C\) be the complementary
\(x_i\)--\(x_j\) and \(y_i\)--\(y_j\) paths in \(R_{ij}\) and \(C\),
respectively. Then
\[
R':=P_i\cup\overline D_C\cup P_j\cup\overline D_F
\]
is a cycle. For each protected block \(B\) whose support is disjoint
from \(D_F\) and which contains neither \(x_i\) nor \(x_j\), add to
\(R'\) the unused switch sides in \(B\) and keep \(B\) as a protected
block. This gives a protected switcher \(F'\) with longest route
\(R'\).

Every protected block not kept in \(F'\) either has its support
contained in \(I_F\), has an endvertex of \(I_F\) in its support, or
contains \(x_i\) or \(x_j\). The first class has total flexibility at
most \(\omega(I_F)\), while the other two classes contain at most
four blocks. Hence
$
\Omega(F')
\ge\Omega-\omega(I_F)-4\omega
\ge\left(1-\frac{1000}{q}\right)\Omega-4\omega.
$

Since each joining path has at least one edge,
\[
\begin{aligned}
\ell(F')
&\ge\ell(R_{ij})-\ell(D_F)+c-\ell(D_C)+2\\
&\ge\ell-\frac{1000\ell}{q}-4b
+c-\frac{c}{100}+1\\
&>\left(1-\frac{1000}{q}\right)\ell+0.99c-4b.
\end{aligned}
\]
Also, \(\ell(F')\le\ell+c+2t\).

Let \(\mathcal Q_0\) be the collection of maximal nontrivial subpaths of members of \(\mathcal Q\) that are contained in \(\overline D_F\). Since the change from \(R\)
to \(R_{ij}\) occurs only inside protected blocks, which are disjoint
from the free paths,
$
\sum_{Q'\in\mathcal Q_0}\ell(Q')\ge W-w(I_F).
$
Let \(Q_C\) be the path formed by \(P_i\), \(\overline D_C\), and
\(P_j\), with its first and last edges removed. Then \(Q_C\) is free
and \(\ell(Q_C)\ge0.99c-1\). Thus, for
\(\mathcal Q':=\mathcal Q_0\cup\{Q_C\}\), equation \eqref{w} gives
$
W(\mathcal Q')
\ge W-w(I_F)+0.99c-1
\ge\left(1-\frac{1000}{q}\right)W+0.99c-1.
$
Removing \(D_F\) creates at most two additional free subpaths, so
\(p(\mathcal Q')\le p+3\). Every protected block of \(F'\) is an
unchanged protected block of \(F\).
\end{proof}

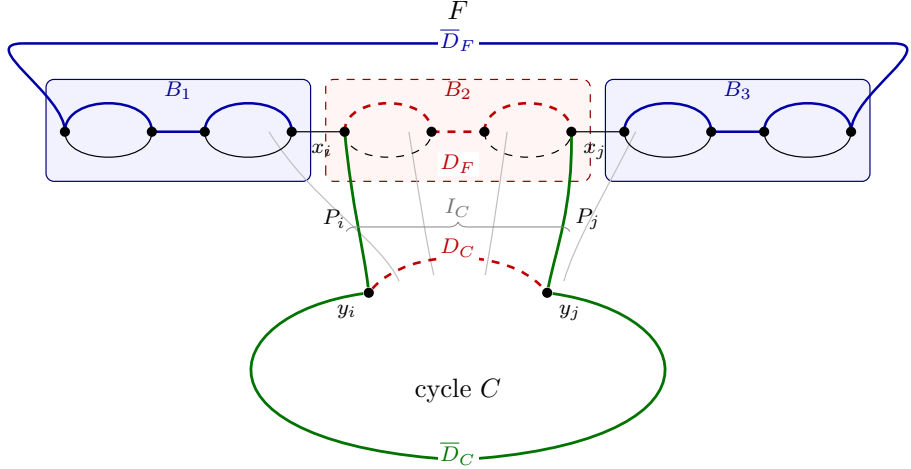
\begin{figure}[t]
\centering
\begin{tikzpicture}[
  x=1cm,
  y=.9cm,
  every node/.style={font=\scriptsize},
  vertex/.style={circle,fill=black,inner sep=1.35pt},
  edge/.style={draw=black,line width=.45pt},
  kept/.style={draw=blue!65!black,line width=.95pt},
  free/.style={draw=green!45!black,line width=1.05pt},
  omitted/.style={draw=red!75!black,dashed,line width=1pt},
  candidate/.style={draw=black!25,line width=.45pt},
  survivebox/.style={
    draw=blue!45!black,
    fill=blue!5,
    rounded corners=3pt,
    line width=.45pt
  },
  discardbox/.style={
    draw=red!55!black,
    fill=red!4,
    dashed,
    rounded corners=3pt,
    line width=.45pt
  }
]

\foreach \i/\x in {1/0,2/1.85,3/3.70,4/5.55,5/7.40,6/9.25}{
  \coordinate (v\i) at (\x,4.25);
  \coordinate (w\i) at ({\x+1.15},4.25);
}

\begin{scope}[on background layer]
  \draw[survivebox] (-.25,3.52) rectangle (3.25,5.02);
  \draw[discardbox] (3.45,3.52) rectangle (6.95,5.02);
  \draw[survivebox] (7.15,3.52) rectangle (10.65,5.02);
\end{scope}

\node[font=\small] at (5.20,6.05) {\(F\)};
\node[text=blue!55!black] at (1.50,4.86) {\(B_1\)};
\node[text=red!65!black] at (5.20,4.86) {\(B_2\)};
\node[text=blue!55!black] at (8.90,4.86) {\(B_3\)};

\foreach \i in {1,2,5,6}{
  \draw[kept] (v\i) .. controls +(0,.56) and +(0,.56) .. (w\i);
  \draw[edge] (v\i) .. controls +(0,-.50) and +(0,-.50) .. (w\i);
}

\foreach \i in {3,4}{
  \draw[omitted] (v\i) .. controls +(0,.56) and +(0,.56) .. (w\i);
  \draw[edge,dashed] (v\i) .. controls +(0,-.50) and +(0,-.50) .. (w\i);
}

\draw[kept] (w1)--(v2);
\draw[edge] (w2)--(v3);
\draw[omitted] (w3)--(v4);
\draw[edge] (w4)--(v5);
\draw[kept] (w5)--(v6);

\draw[kept]
  (w6) .. controls +(0,.55) and +(.65,0) .. (10.95,5.55)
       -- (-.55,5.55)
       .. controls +(-.65,0) and +(0,.55) .. (v1);

\node[
  text=blue!65!black,
  fill=white,
  inner sep=1pt
] at (5.20,5.6) {\(\overline D_F\)};

\coordinate (xi) at (v3);
\coordinate (xj) at (w4);
\node[vertex] (yi) at (4.02,1.88) {};
\node[vertex] (yj) at (6.38,1.88) {};

\begin{scope}[on background layer]
  \draw[candidate]
    (2.70,4.25) .. controls (3.25,3.25) and (4.18,2.65) .. (4.42,2.05);
  \draw[candidate]
    (4.55,4.25) .. controls (4.70,3.22) and (4.78,2.62) .. (4.88,2.14);
  \draw[candidate]
    (5.85,4.25) .. controls (5.72,3.22) and (5.64,2.62) .. (5.56,2.14);
  \draw[candidate]
    (7.55,4.25) .. controls (7.10,3.25) and (6.76,2.65) .. (6.60,2.05);
\end{scope}

\draw[free]
  (xi) .. controls (3.80,3.12) and (3.95,2.55) .. (yi)
  node[pos=.47,left,text=black] {\(P_i\)};

\draw[free]
  (xj) .. controls (6.72,3.12) and (6.50,2.55) .. (yj)
  node[pos=.47,right,text=black] {\(P_j\)};

\draw[omitted]
  (yi) .. controls (4.55,2.55) and (5.85,2.55) .. (yj);

\draw[free]
  (yj) .. controls (8.45,1.55) and (8.80,-.35) .. (5.20,-.58)
       .. controls (1.60,-.35) and (1.95,1.55) .. (yi);

\node[font=\small] at (5.20,.45) {cycle \(C\)};

\node[
  text=green!45!black,
  fill=white,
  inner sep=1pt
] at (5.20,-.48) {\(\overline D_C\)};

\node[
  text=red!75!black,
  fill=white,
  inner sep=1pt
] at (5.20,3.78) {\(D_F\)};

\node[
  text=red!75!black,
  fill=white,
  inner sep=1pt
] at (5.20,2.54) {\(D_C\)};

\draw[
  black!45,
  decorate,
  decoration={brace,amplitude=3.5pt}
]
  (3.72,2.78)--(6.68,2.78)
  node[midway,yshift=9pt,text=black!55] {\(I_C\)};

\node[below left=1pt] at (xi) {\(x_i\)};
\node[below right=1pt] at (xj) {\(x_j\)};
\node[below left=1pt] at (yi) {\(y_i\)};
\node[below right=1pt] at (yj) {\(y_j\)};

\foreach \p in {v1,w1,v2,w2,v3,w3,v4,w4,v5,w5,v6,w6}{
  \node[vertex] at (\p) {};
}

\end{tikzpicture}
\caption{The paths \(P_i\) and \(P_j\) are selected from the family of
vertex-disjoint joining paths. The subpaths \(D_F\) and \(D_C\) are
deleted, while
\(\overline D_F\cup P_i\cup\overline D_C\cup P_j\)
forms the longest route \(R'\) of \(F'\). The protected blocks
\(B_1\) and \(B_3\) are kept, whereas \(B_2\) is discarded.}
\label{fig:protected-joining}
\end{figure}

\section{Constructing protected switchers}
We now use the expansion results from the preceding section to construct switchers whose length and flexibility remain large under successive joining steps.  

Throughout this section, let $H=K_{m_1,\ldots,m_k}$, where
$m_1\le\cdots\le m_k$, and set $h:=|H|$ and $m:=h/k$. We assume that
$k\ge k_0$ and $m\le k^{22}$, where \(k_0\) is a sufficiently large
absolute constant. In particular, \(k\le h\le k^{23}\). The case when
\(k<k_0\) follows from Theorem~\ref{thm:HHKL}, while the case when
\(m>k^{22}\) follows from Corollary~\ref{cor:PS-large}.

We first show how the expansion properties can be used to construct switchers.
\begin{lemma}\label{lem:flexible-switcher}
There is an absolute constant \(C_0>0\) such that, for every integer
\(M\ge1\), there exists \(k_0=k_0(M)\) for which the following holds
whenever \(k\ge k_0\). Let \(L\) be a complete \(a\)-partite graph with \(2\le a\le k\) and \(|L|\le h\). Suppose that \(\ol G\) is \(L\)-free, \(|G|\ge(a-1)n/10\), and \(n\ge C_0h\). Then \(G\) contains a protected switcher \(F^*\) such that \(|F^*|\le C_0h\) and \(\Omega(F^*)\ge40h\). Moreover, \(F^*\) has exactly \(M\) protected blocks. If \(b_0\) and \(\omega_0\) denote the maximum order and maximum flexibility of a protected block, respectively, then \(b_0,\omega_0\le C_0h/M\).
\end{lemma}

\begin{proof}
We first replace \(L\) by a complete multipartite graph whose parts have suitable orders. Set \(d:=\lceil |L|/a\rceil\).  Enlarge each part of order less than \(d\), and divide every resulting part of order \(x\ge d\) as evenly as possible into \(\lceil x/(2d)\rceil\) parts. Adding all edges between distinct new parts gives a complete \(q\)-partite graph \(\mathcal J\supseteq L\). For convenience, we call this operation a \textit{normalization}.
%{\color{red} \textit{superlarge}.}
Every part of \(\mathcal J\) has order in \([d,2d]\), while \(a\le q\le2a\), \(qd\le4h\), and \(|\mathcal J|\le3|L|\). In particular, \(\ol G\) is \(\mathcal J\)-free.

Choose an absolute constant \(D\) sufficiently large. Set \(\beta:=1000Dh\) and \(R:=\lfloor\beta/1000\rfloor\). Thus
\(\beta\ge 1000Dh>qd\).
Since
\(p_{q,q}(d,2\beta)\le8K\beta q\le16000KDha\le32000KD(a-1)h\),
we have \(|G|\ge p_{q,q}(d,2\beta)\).
By Lemma~\ref{lem:expanding-subgraph}, there exists an induced subgraph \(F_0\subseteq G\) and a
part-subgraph \(J_0\subseteq\mathcal J\) with \(r_0\ge2\) parts such that 
\begin{itemize}
    \item[(a)]  \(\ol{F_0}\) is \(J_0\)-free,
    \item[(b)]  $p_{r_0,q}(d,2\beta)-2d \le |F_0| \le p_{r_0,q}(d,2\beta)$,
    \item[(c)]  \(F_0\) multi-\((d,r_0,q,2\beta,20)\)-expands.
\end{itemize}
The graph \(F_0\) is not bipartite. Indeed, otherwise one bipartition class
\(X\) has order at least \(|F_0|/2\). Since
\(p_{r_0,q}(d,2\beta)-2d
 =\left\lceil2K\beta\left(2r_0-3+
 \frac{(r_0-1)\log r_0}{\log(2q)}\right)\right\rceil
 \ge2K\beta(2r_0-3)
 \ge K\beta r_0\), we have
\(|X|\ge K\beta r_0/2>2dr_0\ge|J_0|\). The set \(X\) is a clique in
\(\ol{F_0}\), and hence contains \(J_0\), a contradiction.
By Lemma~\ref{lem:compact-odd}, we obtain an odd cycle \(C_1\) such that \(|C_1|\le2E(r_0,q,2\beta)+1\), and \(F_0\)
multi-\((d,r_0,q,2\beta,16)\)-expands into \(W:=V(F_0)\setminus V(C_1)\). Thus, \(F_0\) multi-\((d,r_0,q,\beta,16)\)-expands into \(W\). Indeed, the small-set condition is unchanged below \(\beta\), while for \(\beta\le|S|\le2\beta\) we have
$
 |N_{F_0}(S)|
 \ge16|S|
 \ge\phi_{r_0,q,\beta}(|S|)|S|+10\beta,
$
since \(\phi_{r_0,q,\beta}\le3\eta<1\). For \(|S|\ge2\beta\), since
\(p_{t,q}(d,2\beta)\ge p_{t,q}(d,\beta)\), we have
\(\phi_{r_0,q,2\beta}(|S|)\ge\phi_{r_0,q,\beta}(|S|)\). Thus 
\begin{itemize}
    \item[(d)]  \(F_0\) multi-\((d,r_0,q,\beta,16)\)-expands into \(W\).
\end{itemize}

Set \(\mu:=10E(r_0,q,2\beta)\) and write
\(E:=E(r_0,q,\beta)\). We construct an open chain whose switches are
divided into \(M\) consecutive complete blocks. Every short odd cycle used below has order at most \(\mu\), and every path joining consecutive switches has length at most \(\mu\).
Since \(q\le2k\), \(\beta=1000Dh\), and \(k\leq h\le k^{23}\), we have
\(\mu=O(\log(2k)\log\log(4k))=o(h)\). 

Choose almost-antipodal vertices \(v_1,w_1\) on \(C_1\). The two
\(v_1\)--\(w_1\) paths on \(C_1\) differ in length by one, so \(\delta(C_1)=1\). Moreover, \(|C_1|\le\mu\). 
Since \(R\le\beta/2\), Lemma~\ref{lem:multi-root} gives disjoint
\((v_1,R)\)-rooted set $A_1$ and \((w_1,R)\)-rooted set \(B_1\subseteq
W\cup\{v_1,w_1\}\). We keep \(A_1\) as the initial rooted set and use \(B_1\) as the active end set.
We extend this chain by adding switches \(C_2,C_3,\ldots\) and paths
joining \(w_i\) to \(v_{i+1}\). There are two ways to extend.
We first add weight-one switches until the total weight is at least \(8\mu\). Thereafter, we add switches whose weights are a fixed positive proportion of the current total weight. This preserves completeness and makes the flexibility grow geometrically. We stop when the flexibility first reaches
\(\nu:=\lceil40h/M\rceil\). 

We perform the following construction while the switches and joining
paths chosen so far have total order at most \(Dh\). We shall verify
below that this condition holds throughout the construction of all
\(M\) protected blocks.

\smallskip
\noindent
\textbf{Step 1. Unit extension.}

Suppose that \(C_i\) is the last switch constructed, with second pivotal vertex \(w_i\), and let \(B_i\) be the active
\((w_i,R)\)-rooted set. Let \(H_i\) be the subgraph of \(F_0\) induced by the vertices not used by the switches and joining paths already chosen, or by the initial and active rooted sets. The order estimate proved below gives \(|H_i|\ge p_{r_0,q}(d,\beta)\), while \(\ol{H_i}\) does not contain \(J_0\). Lemma~\ref{lem:expanding-subgraph}
gives an induced subgraph \(F_i\subseteq H_i\) and a part-subgraph
\(J_i\subseteq J_0\) with \(r_i\ge2\) parts such that
\(\ol{F_i}\) does not contain \(J_i\),
\(p_{r_i,q}(d,\beta)-2d\le|F_i|\le p_{r_i,q}(d,\beta)\), and \(F_i\)
multi-\((d,r_i,q,\beta,20)\)-expands.

Note that \(F_i\) is non-bipartite by the similar bipartition argument for \(F_0\). Since
\(|F_i|\le p_{r_i,q}(d,\beta)\le p_{r_i,q}(d,2\beta)\),
Lemma~\ref{lem:compact-odd} gives an odd cycle \(C_{i+1}\) such that
$
|C_{i+1}|\le2E(r_i,q,\beta)+1\le\mu
$
and \(F_i\) multi-\((d,r_i,q,\beta,16)\)-expands into
\(W_i:=V(F_i)\setminus V(C_{i+1})\). Choose almost-antipodal vertices \(v_{i+1},w_{i+1}\in V(C_{i+1})\). Then \(C_{i+1}\) is a weight-one switch. Since \(R\le\beta/2\), Lemma~\ref{lem:multi-root}, applied in \(F_i\) with target set \(W_i\), gives disjoint
\((v_{i+1},R)\)-rooted set $A_{i+1}$ and \((w_{i+1},R)\)-rooted set
\(B_{i+1}\).

Apply Lemma~\ref{lem:short-path} in \(F_0\) from \(B_i\) to
\(A_{i+1}\), avoiding all previously chosen switches and joining paths
except for \(w_i\), as well as the initial rooted set,
\(V(C_{i+1})\setminus\{v_{i+1}\}\), and \(B_{i+1}\). The previously
chosen switches and joining paths use at most \(Dh\) vertices, and
\(\mu\le R\). Hence the forbidden set has order and
\(W\)-intersection at most
$
Dh+2R+\mu\le Dh+3R\le10R<5\beta,
$
and \(10R<14R\). Thus Lemma~\ref{lem:short-path} applies.

Orient the resulting path from \(B_i\) to \(A_{i+1}\), take the
subpath from its last vertex in \(B_i\) to its first vertex in
\(A_{i+1}\), and extend its ends inside these rooted sets to \(w_i\)
and \(v_{i+1}\). This gives a \(w_i\)--\(v_{i+1}\) path of length at
most \(E+2\lceil\log_8R\rceil\le\mu\). This path joins
\(C_{i+1}\) to the preceding switches, and \(B_{i+1}\) becomes the
active rooted set.

\smallskip
Once the total weight \(S\) of the current block reaches \(8\mu\), we stop using unit extensions. Instead, we seek a new switch of weight \(\delta\) satisfying
$
  \frac{3S}{8}\le \delta\le \frac S2+1.
$
Since the current block is complete and \(\delta\le S+1\), adding such a switch preserves completeness. Moreover, the total weight increases from \(S\) to at least \(11S/8\). We call this operation a \emph{geometric extension}, and repeat it until the total weight first reaches \(\nu\).

\smallskip
\noindent
\textbf{Step 2. Geometric extension.}

After the unit extensions, let \(S\) be the total weight of the
current block. Then \(S\ge8\mu\), and the subset sums of its switch weights contain every integer in \([0,S]\). Applying the same argument as in Step~1 to the unused vertices of \(F_0\), we obtain an induced subgraph \(F'\) and an integer \(r'\ge2\) such that \(F'\) multi-\((d,r',q,\beta,20)\)-expands. Choose distinct vertices \(a,v,w,b\in V(F')\) such that \(av,vw,wb\) are edges. Since \(R\le\beta/4\), Lemma~\ref{lem:multi-root} gives pairwise disjoint
rooted sets \(A_a,A_v,A_w,A_b\) of order \(R\) at \(a,v,w,b\),
respectively.

Set \(T:=\lfloor S/2\rfloor-4E\). Since \(S\ge8\mu\) and
\(E\le\mu/10\), we have \(T\ge18\mu/5>0\). Let \(Z\) consist of all vertices in the switches and joining paths chosen so far, together with the initial rooted set, the active rooted set, and \(A_v\cup A_w\). We have \(|Z|\le Dh+4R\) and
\(|Z\cap W|\le Dh+4R\). Since \(S<\nu\),
\(T+4E=\lfloor S/2\rfloor\le\nu/2\le R\), and hence
$
|Z|+T+4E+2R\le Dh+7R\le5\beta,
$
and
$
|Z\cap W|+T+4E\le Dh+5R\le12R.
$
Lemma~\ref{lem:coarse-path} gives an \(a\)--\(b\) path \(P\)
avoiding \(Z\) with \(T\le\ell(P)\le T+4E\).

The cycle \(vaPbw v\), with pivotal vertices \(v,w\), is a
switch of weight \(\delta=\ell(P)+1\).
Since \(4E\le S/8\),
\begin{equation}\label{eq:geometric-weight}
 \frac{3S}{8}\le\delta\le\frac S2+1.
\end{equation}
The subset sums after adding this switch contain the integer intervals
\([0,S]\) and \([\delta,S+\delta]\). Since \(\delta\le S+1\), their
union contains every integer in \([0,S+\delta]\). Thus the resulting protected block is complete, and its total weight
satisfies
\[
\frac{11S}{8}\le S+\delta\le\frac{3S}{2}+1.
\]

Let \(z\) be the second pivotal vertex of the last switch
previously chosen, and let \(B_i\) be the active
\((z,R)\)-rooted set. The new switch has order
\(\ell(P)+3\le\nu/2+3\). Apply Lemma~\ref{lem:short-path} in \(F_0\)
from \(B_i\) to \(A_v\), avoiding all previously chosen switches and
joining paths except for \(z\), the initial rooted set, the new switch except for \(v\), and \(A_w\). The forbidden set has order and \(W\)-intersection at most
$
Dh+2R+\frac{\nu}{2}+3\le Dh+3R\le10R.
$
Thus there is a \(z\)--\(v\) path of length at most \(\mu\), internally disjoint from all previously chosen vertices and from the new switch except at \(v\). The rooted set \(A_w\) becomes the active rooted set.
\smallskip

Partition the switches into consecutive protected blocks as follows.
The first block begins with \(C_1\), and each subsequent block begins with a weight-one switch obtained by a unit extension. Since \(M\) is fixed and \(\mu^2=o(h)\), we may increase \(k_0\) so that
\(8\mu<\nu\) and \(\mu^2\le h/M\).
Within each block, use unit extensions until its total weight reaches \(\lceil8\mu\rceil\), and then use geometric extensions until its weight first reaches \(\nu\). Completeness is preserved at every step. Before the final geometric extension the total weight is less than \(\nu\), so \eqref{eq:geometric-weight} shows that the final flexibility lies in \([\nu,3\nu/2+1]\).

The unit extensions add \(O(\mu)\) switches, and each such switch
together with its joining path has order \(O(\mu)\). Their total order
is therefore \(O(\mu^2)\). During the geometric extensions, the total
weight increases by a factor of at least \(11/8\), while the order of
each new switch is \(O(S)\), where \(S\) is the weight before that
extension. Hence the new switches have total order \(O(\nu)\).

There are \(O(\log(\nu/\mu))=O(\log h)\) geometric extensions. Each
corresponding joining path has length \(O(E)=O(\mu)\), so these paths
have total order \(O(\mu\log h)=O(\mu^2)\). Consequently, each
protected block has order
\begin{equation}\label{eq:block-apriori}
 O(\nu+\mu^2)=O(h/M),
\end{equation}
and flexibility \(O(h/M)\). All implicit constants are absolute and
independent of \(M\).

After completing one protected block, begin the next with a unit
extension. The supports of the resulting blocks are consecutive and
pairwise edge-disjoint. By \eqref{eq:block-apriori}, after \(M\)
blocks the switches and joining paths chosen have total order
$
O(M\nu+M\mu^2)=O(h),
$
since \(M\nu\le40h+M\le41h\) and \(M\mu^2\le h\). Choose \(D\)
so that this order is at most \(Dh\).

At any stage, the only other vertices in use lie in the initial and
active rooted sets. Hence fewer than \(Dh+2R\le10R<\beta\) vertices
have been removed from \(F_0\). Moreover,
$
p_{r_0,q}(d,2\beta)-2d-p_{r_0,q}(d,\beta)
\ge K\beta-2d-2
\ge100\beta.
$
Thus the subgraph induced by the unused vertices has order at least
\(p_{r_0,q}(d,\beta)\), and its complement does not contain \(J_0\).
Lemma~\ref{lem:expanding-subgraph} therefore gives, at each step, an
induced subgraph which multi-\((d,r',q,\beta,20)\)-expands for some
\(r'\ge2\). It is non-bipartite by the
same argument used for \(F_0\). The size estimates in Steps~1 and~2
now show that all applications of Lemmas~\ref{lem:short-path} and
\ref{lem:coarse-path} are valid.

After completing the \(M\)th protected block, apply Lemma~\ref{lem:short-path} from the active rooted set to the initial
rooted set, avoiding every previously chosen switch and joining path
outside the active and initial rooted sets. The forbidden set has order at most
\(Dh\le5\beta\) and \(W\)-intersection at most \(Dh\le7R<14R\). The subpath from its last vertex in the active rooted set to its first vertex in the initial rooted set, extended inside these two sets, joins the last switch to the first and gives a protected switcher \(F^*\).

Each protected block has flexibility at least \(\nu\), and hence
$
\Omega(F^*)\ge M\nu\ge40h.
$
The preceding order estimate gives \(|F^*|=O(h)\), while
\eqref{eq:block-apriori} and the flexibility bound give
\(b_0,\omega_0=O(h/M)\). Increasing \(C_0\) yields
$
|F^*|\le C_0h,
$
and
$
b_0,\omega_0\le\frac{C_0h}{M}.
$
The constant \(C_0\) is absolute and independent of \(M\).
\end{proof}
\begin{remark}
The unit-extension step in the proof of Lemma \ref{lem:flexible-switcher} is essential and cannot be omitted. Indeed, Lemma~\ref{lem:coarse-path} controls the length of the path used in a geometric extension only up to an additive error of \(4E\), so the weight of the resulting switch cannot be prescribed exactly. When the current total weight \(S\) is small, this error may prevent the new weight from being at most \(S+1\), which is necessary to preserve completeness. 
\end{remark}
The next lemma provides the long cycles used in the extension step.
\begin{lemma}\label{lem:long-cycle}
There is an absolute constant \(C_1>0\) such that the following holds.
Let \(L\) be a complete \(a\)-partite graph with \(2\le a\le k\) and
\(|L|\le h\). Suppose that \(\ol G\) is \(L\)-free,
\(|G|\ge(a-1)n/10\), and \(n\ge C_1h\).
Then \(G\) contains a cycle \(C\) with
\begin{equation}\label{eq:long-cycle-range}
\frac{n}{10^{10}}\le|C|\le\frac{2n}{10^{10}}.
\end{equation}
\end{lemma}

\begin{proof}
Set \(L_0:=n/10^{10}\), \(d:=\lceil |L|/a\rceil\) and \(\beta:=n/10^7\). Let \(\mathcal J\) be a complete \(q\)-partite graph with \(q\le2a\) and \(qd\le4h\) by normalizing $L$. By increasing \(C_1\), we
have \(\beta\ge qd\). By calculation, we have
\(p_{q,q}(d,2\beta)\le8K\beta q\le16Kan/10^7\le(a-1)n/10\).
Lemma~\ref{lem:expanding-subgraph} therefore gives
an induced subgraph \(F_0\subseteq G\) and a part-subgraph
\(J_0\subseteq\mathcal J\) with \(r_0\ge2\) parts such that
\(p_{r_0,q}(d,2\beta)-2d\le|F_0|\le p_{r_0,q}(d,2\beta)\),
\(\ol{F_0}\) is \(J_0\)-free, and \(F_0\) multi-\((d,r_0,q,2\beta,20)\)-expands and hence also multi-\((d,r_0,q,\beta,16)\)-expands.

Since \(r_0\le q\le2k\le2h\le2n/C_1\), we have
\(E(r_0,q,\beta)=O(\log n\log\log n)=o(n)\) and hence \(4E(r_0,q,\beta)\le\min\{\lfloor L_0\rfloor,L_0/2\}\).
Choose a path \(avwb\) on four distinct
vertices. Since \(4\lfloor L_0\rfloor\le\beta\), Lemma~\ref{lem:multi-root} gives pairwise disjoint rooted sets \(A_a,A_v,A_w,A_b\) of order \(\lfloor L_0\rfloor\) at \(a,v,w,b\), respectively.

Apply Lemma~\ref{lem:coarse-path} from \(A_a\) to \(A_b\), excluding \(Z:=A_v\cup A_w\), with \(T:=\lceil L_0\rceil-3\). We have \(|Z|=2\lfloor L_0\rfloor\) and \(|Z\cap W|\le2\lfloor L_0\rfloor\). Moreover, \(4\lfloor L_0\rfloor+T+4E(r_0,q,\beta)\le5\beta\) and \(2\lfloor L_0\rfloor+T+4E(r_0,q,\beta)\le12\lfloor L_0\rfloor\). Hence there is an \(a\)--\(b\) path \(P\), disjoint from \(A_v\cup A_w\), such that
$
T\le\ell(P)\le T+4E(r_0,q,\beta).
$

The cycle \(C:=avwbPa\) has length \(\ell(P)+3\). Thus \(|C|\ge\lceil L_0\rceil\ge L_0\), while
$
|C|\le\lceil L_0\rceil+4E(r_0,q,\beta)\le L_0+1+\frac{L_0}{2}\le2L_0.
$
Since \(L_0=n/10^{10}\), we obtain \(n/10^{10}\le|C|\le2n/10^{10}\).
\end{proof}

For the rest of the paper, fix a sufficiently large absolute integer \(M\) such that \(M\ge2\times10^{15}C_0\), and fix \(k_0=k_0(M)\) sufficiently large, increasing it as necessary below.
Recall that \(H=K_{m_1,\ldots,m_k}\), where \(m_1\le\cdots\le m_k\), and that \(h=|H|\). In particular, \(m_2\) denotes the order of the second part in this nondecreasing ordering.

\begin{lemma}\label{lem:local-master}
There is an absolute constant $C_2$ such that the following holds. Let $H'=K_{b_1,\ldots,b_{s+1}}$ be a part-subgraph of $H$, where $1\le s\le k-1$. Suppose that $n\ge C_2h$, $|G|\ge sn-n/10$ and $\ol G\ \text{is }H'\text{-free}$.
Assume further that any two disjoint vertex sets of order at least $m_2+10^{18}$ are connected by at least $10^{18}$ pairwise vertex-disjoint
paths. Then:
\begin{enumerate}[label=\textup{(\roman*)}]
\item if $s\ge2$, then $G$ contains $C_n$;
\item if $s=1$, then $G$ contains a protected switcher $J$ satisfying $0.64n<\ell(J)<0.70n$ and $20h\le\Omega(J)\le C_0h$.
Every protected block of $J$ has order and flexibility at most
$C_0h/M$.
\end{enumerate}
\end{lemma}

\begin{proof}
Applying Lemma~\ref{lem:flexible-switcher} with \(L=H'\), we obtain a protected switcher \(J_0\) such that \(|J_0|\le C_0h\),
\(\Omega(J_0)\ge40h\), and
\(b_0,\omega_0\le C_0h/M\). Since
\(\Omega(J_0)\le|J_0|\) and \(m_2\le h\), increasing \(k_0\) gives
\(|J_0|\ge m_2+10^{18}\) and \(C_0h/M\ge1\). Moreover,
\(n\ge C_2h\) and \(C_2>100C_0\) give \(|J_0|<n/100\). Replacing
\(b_0\) by \(\max\{b_0,1\}\), we may assume that
\(1\le b_0\le C_0h/M\).

We construct \(J_i\) iteratively. Given \(J_{i-1}\), apply
Lemma~\ref{lem:long-cycle} in \(G-V(J_{i-1})\) to obtain a cycle
\(C_i\), and connect \(C_i\) to \(J_{i-1}\) by two paths to form $J_i$ by Lemma \ref{lem:protected-splice}. Write \(c_i:=|C_i|\), and let \(g\) be the least index such
that
\[
\sum_{i=1}^g c_i\ge
\begin{cases}
4n/3,&s\ge2,\\
2n/3,&s=1.
\end{cases}
\]
Since \(c_i\ge n/10^{10}\), we have
\begin{equation}\label{eq:number-long-cycles}
g\le2\times10^{10}.
\end{equation}

Before \(C_i\) is chosen,
$
|J_{i-1}|
\le C_0h+\sum_{j<i}c_j+2(i-1)(h+k),
$
as each connecting step uses two paths of length at most \(h+k\) by Lemma \ref{lem:shortcut}(i).
If \(s=1\), then \(\sum_{j<i}c_j<2n/3\), and hence
\[
|G-V(J_{i-1})|
\ge0.9n-\frac{2n}{3}-C_0h-4\times10^{10}(h+k)
>\frac n{10}.
\]
If \(s\ge2\), then \(\sum_{j<i}c_j<4n/3\), and
\[
|G-V(J_{i-1})|
\ge\left(s-\frac1{10}\right)n-\frac{4n}{3}
-C_0h-4\times10^{10}(h+k)
>\frac{sn}{10}.
\]
The last inequalities follow from \(k\le h\) and \(n\ge C_2h\), after
choosing \(C_2\) sufficiently large. Thus
Lemma~\ref{lem:long-cycle} applies at every stage.
Since \(c_g\le2n/10^{10}\),
$
\frac{4n}{3}\le\sum_{i=1}^g c_i
<\frac{4n}{3}+\frac{2n}{10^{10}}
$
if
$
s\ge2,
$
while
$
\frac{2n}{3}\le\sum_{i=1}^g c_i
<\frac{2n}{3}+\frac{2n}{10^{10}}
$
if
$
s=1.
$

We claim that at every stage $\Omega(J_{i-1})\ge20h$ will hold. Hence $|J_{i-1}|\ge m_2+10^{18}$, and
$c_i\ge n/10^{10}\ge m_2+10^{18}$. The connectivity hypothesis gives
$10^{18}$ vertex-disjoint joining paths. For each path, take the subpath from its last vertex in $J_{i-1}$ to its first vertex in $C_i$, and replace this subpath by a shortest path in its own vertex set. Lemma~\ref{lem:shortcut}(i), applied to $H'$, bounds its length by $|H'|+(s+1)-3\le h+k$. Since each cycle is chosen immediately before it is joined to the current switcher, no joining path meets a cycle chosen later.
Let $L_i=\ell(J_i)$ and $\Omega_i=\Omega(J_i)$. Take
$\mathcal Q_0=\varnothing$.  For $i\ge1$, let $\mathcal Q_i$ be the free family obtained from $\mathcal Q_{i-1}$ by Lemma~\ref{lem:protected-splice},
and set $W_i=W(\mathcal Q_i)$ and $p_i=p(\mathcal Q_i)$.
Thus
\begin{align}
 L_i&\ge(1-10^{-15})L_{i-1}+0.99c_i-4b_0,
 \label{eq:L-rec}\\
 L_i&\le L_{i-1}+c_i+2(h+k),
 \label{eq:L-up-rec}\\
 \Omega_i&\ge(1-10^{-15})\Omega_{i-1}-4\omega_0,
 \label{eq:Omega-rec}\\
 W_i&\ge(1-10^{-15})W_{i-1}+0.99c_i-1,
 \label{eq:W-rec}\\
 p_i&\le p_{i-1}+3.
 \label{eq:p-rec}
\end{align}
Iterating \eqref{eq:Omega-rec}, and using
\((1-10^{-15})^i\ge1-i10^{-15}\), gives
$
\Omega_i
\ge(1-i10^{-15})\Omega_0-4i\omega_0
$
for every \(0\le i\le g\). By
\eqref{eq:number-long-cycles},
\(i10^{-15}\le2\times10^{-5}\), while
$
4i\omega_0
\le\frac h{100}.
$
Since \(\Omega_0\ge40h\), it follows that
$
\Omega_i
\ge(1-2\times10^{-5})40h-\frac h{100}
>20h.
$
Thus \(\Omega(J_i)\ge20h\) throughout the construction.

Since \(10^{-15}g\le2\times10^{-5}\), we have
\((1-10^{-15})^{g-i}\ge1-2\times10^{-5}\) for every \(1\le i\le g\).
Iterating \eqref{eq:L-rec} and \eqref{eq:W-rec}, and using
\(b_0\le C_0h/M\), gives
\[
L_g\ge0.99(1-2\times10^{-5})\sum_{i=1}^g c_i-4gb_0
\ge0.989\sum_{i=1}^g c_i-\frac{8\times10^{10}C_0h}{M}
\]
and
\[
W_g\ge0.99(1-2\times10^{-5})\sum_{i=1}^g c_i-g
\ge0.989\sum_{i=1}^g c_i-2\times10^{10}.
\]
Also, \eqref{eq:p-rec} gives \(p_g\le3g\le6\times10^{10}\).
After fixing \(M\) and increasing \(C_2\), the minimality of \(g\) yields \(L_g>1.30n\) and \(W_g>1.29n\) when \(s\ge2\), and \(L_g>0.64n\) when \(s=1\).
Iterating \eqref{eq:L-up-rec} gives
\(L_g\le C_0h+\sum_{i=1}^g c_i+4\times10^{10}(h+k)\). By the minimality of \(g\) and \(c_g\le2n/10^{10}\), \(L_g<1.36n\) when \(s\ge2\) and \(L_g<0.70n\) when \(s=1\), after increasing \(C_2\).
Every switch of \(J_g\) belongs to \(J_0\), and every protected block of \(J_g\) is an unchanged protected block of \(J_0\). Hence
\(\Omega(J_g)\le\Omega(J_0)\le|J_0|\le C_0h\), while the bounds
\(b_0,\omega_0\le C_0h/M\) remain valid. This proves~\textup{(ii)}.

Suppose that \(s\ge2\), and set \(j:=|H'|\), \(a:=s+1\), and
\(d_0:=j+a-2\le h+k\). By \eqref{eq:p-rec} and
\eqref{eq:number-long-cycles}, \(p_g\le6\times10^{10}\). Since
\(d_0\le h+k\le2h\) and \(n\ge C_2h\), we have \(p_gd_0\le n/100\).
Moreover, \(\sum_{i=1}^g c_i\le4n/3+2n/10^{10}\). The estimates for
\(W_g\) and \(L_g\) therefore give
$
W_g-L_g+n-p_gd_0
\ge n-0.011\sum_{i=1}^g c_i-C_0h
   -4\times10^{10}(h+k)-2\times10^{10}-p_gd_0
>0.
$
Hence
\begin{equation}\label{eq:free-invariant}
W_g\ge L_g-n+p_gd_0.
\end{equation}

The shortening steps modify only free subpaths, so the protected blocks and their total flexibility remain unchanged. If the current longest route has length \(L>n+d_0\), then every route has length at least \(L-C_0h>n+d_0-C_0h\ge j+a\). Moreover,
\eqref{eq:free-invariant} gives \(W>p_gd_0\), so some free subpath has at least \(d_0\) edges.
Apply Lemma~\ref{lem:shortcut}\textup{(ii)} to \(d_0\) consecutive edges of such a subpath. The replacement shortens every route by some
\(t\in[1,d_0-1]\). Thus \(L'=L-t\), \(W'\ge W-t\), and
$
W'\ge L'-n+p_gd_0.
$
Hence \eqref{eq:free-invariant} is preserved, while
\(L'>n+d_0-(d_0-1)>n\). Repeating this step yields
\(n\le L\le n+d_0\).
Finally, \(\Omega(J_g)\ge20h\ge d_0\), and therefore
\(L-\Omega(J_g)\le L-d_0\le n\le L\). Lemma~\ref{lem:interval-routes}
gives a route of length \(n\), proving~\textup{(i)}.
\end{proof}

\section{Proof of Theorem \ref{thm:main}}

Finally we derive Theorem~\ref{thm:main} from the following stability dichotomy.
\subsection{Stability}
The following lemma is a variant of Lemma~6.2 in~\cite{HHKL}. It gives the following dichotomy: Either the problem reduces to a forbidden multipartite graph with one fewer part, or \(G\) is close to the extremal construction, with \(k-1\) large
pairwise anticomplete sets that are internally highly connected.
\begin{lemma}\label{lem:stability}
There is an absolute constant $C_3$ such that the following holds.
Let $H=K_{m_1,\ldots,m_k}$ have order $h$, where
$m_1\le\cdots\le m_k$, $k\ge k_0$, and $h/k\le k^{22}$, and let $z\ge0$.
Let $\widehat H$ be obtained from $H$ by deleting the part of order $m_2$.
Suppose that $G$ is $C_n$-free,
$
 n\ge C_3h,
$
$
|G|\ge(k-1)(n-1)+z,
$
and
$
\ol G\text{ is }H\text{-free}.
$
Then at least one of the following holds.
\begin{enumerate}[label=\textup{(\roman*)}]
\item There is an induced subgraph $G'\subseteq G$ with
$|G'|\ge(k-2)(n-1)+z$ such that $\ol{G'}$ is $\widehat H$-free.
\item There are pairwise disjoint, pairwise anticomplete sets
$A_1,\ldots,A_{k-1}$ such that, for every $i$,
\begin{enumerate}[label=\textup{(\alph*)}]
\item $|A_i|\ge0.95n$;
\item $\ol{G[A_i]}$ is $K_{m_1,m_2}$-free;
\item any two disjoint subsets of $A_i$ of order at least
$m_2+10^{18}$ are joined in $G[A_i]$ by at least $10^{18}$ pairwise
vertex-disjoint paths.
\end{enumerate}
\end{enumerate}
\end{lemma}

\begin{proof}
Assume that \(C_3\) is sufficiently large and that \textup{(i)} does not
hold. Start with \(A=V(G)\) and \(S=\varnothing\). Whenever there is a
set \(X\subseteq A\) with \(|X|<10^{18}\) such that \(G[A-X]\) has more
components of order at least \(m_2\) than \(G[A]\), remove \(X\) from
\(A\) and add it to \(S\). Continue for at most \(k\) rounds, or until no
such set exists. Throughout the process,
\begin{equation}\label{eq:S-bound}
 |S|<10^{18}k.
\end{equation}

For every \(B\subseteq V(G)\) with \(|B|\ge m_2\),
\begin{equation}\label{eq:neighbourhood-n}
 |B\cup N_G(B)|\ge n.
\end{equation}
Otherwise, let \(U:=V(G)\setminus(B\cup N_G(B))\). Then
$
 |U|\ge(k-2)(n-1)+z.
$
Since \textup{(i)} does not hold, \(\ol{G[U]}\) contains a copy of
\(\widehat H\). There are no edges between \(B\) and \(U\) in \(G\), so any \(m_2\) vertices of \(B\), together with this copy of \(\widehat H\), form a copy of \(H\) in \(\ol G\), a contradiction.

For every \(B\subseteq A\) with \(|B|\ge m_2\),
\eqref{eq:S-bound} and \eqref{eq:neighbourhood-n} give
\begin{equation}\label{eq:neighbourhood-A}
 |B\cup N_{G[A]}(B)|\ge n-|S|\ge h.
\end{equation}
Since \(k\ge3\) and \(m_1\le m_2\le\cdots\le m_k\), we have
\(2m_2\le(k-1)m_2\le h\). The components of \(G[A]\) of order less than
\(m_2\) contain fewer than \(m_2\) vertices in total. Otherwise, the
union \(B\) of some of them satisfies \(m_2\le|B|<2m_2\). Since
\(N_{G[A]}(B)=\varnothing\),
$
 |B\cup N_{G[A]}(B)|=|B|<2m_2\le h,
$
a contradiction to \eqref{eq:neighbourhood-A}.

Every remaining component \(C\) of \(G[A]\) has order at least \(m_2\). Since \(N_{G[A]}(C)=\varnothing\), \eqref{eq:neighbourhood-A} gives
$
 |C|\ge n-|S|.
$
Denote these components by \(A_1,\ldots,A_t\).
We have \(t\le k-1\). Otherwise, choose sets of orders
\(m_1,\ldots,m_k\), one from each of \(k\) distinct components. Since
there are no edges between distinct components of \(G[A]\), these sets
form the parts of a copy of \(H\) in \(\ol G\), a contradiction. The
procedure therefore cannot complete \(k\) rounds, since each round
increases the number of components of order at least \(m_2\).

Fix \(i\in[t]\), and let \(B_1,B_2\subseteq A_i\) be disjoint sets of
order at least \(m_2+10^{18}\). If there are fewer than \(10^{18}\)
pairwise vertex-disjoint \(B_1\)--\(B_2\) paths in \(G[A_i]\), then
Menger's theorem gives a set \(X\subseteq A_i\) with \(|X|<10^{18}\)
such that no component of \(G[A_i]-X\) meets both \(B_1\setminus X\) and
\(B_2\setminus X\). For each \(j\in\{1,2\}\), some component meeting
\(B_j\setminus X\) has order at least \(m_2\). Otherwise, the union \(U\)
of some such components satisfies \(m_2\le|U|<2m_2\) and
\(N_G(U)\subseteq X\cup S\), so
$
 |U\cup N_G(U)|<2m_2+10^{18}(k+1)<n,
$
a contradiction to \eqref{eq:neighbourhood-n}. Thus \(G[A_i]-X\) contains two distinct components of order at least \(m_2\), contradicting the choice of \(A\). This proves \textup{(c)}.

\begin{claim}
$t=k-1$.    
\end{claim}
For each \(i\in[t]\), set \(H_i^{(1)}:=\overline {K_{m_{k+1-i}}}\), and, for
\(2\le j\le k-t+1\), set
$$
 H_i^{(j)}
 :=K_{m_{k+1-i},m_{k-t},m_{k-t-1},\ldots,m_{k-t-j+2}}.
$$
Thus \(H_i^{(j)}\) is obtained from \(H_i^{(j-1)}\) by adding a part of
order \(m_{k-t-j+2}\). Since \(|A_i|\ge h\ge m_{k+1-i}\),
\(\ol{G[A_i]}\) contains \(H_i^{(1)}\). Let \(s_i\) be the largest
integer \(j\in[k-t+1]\) for which
\(H_i^{(j)}\subseteq\ol{G[A_i]}\).

Suppose that \(\sum_{i=1}^t s_i\ge k\), and fix a copy of
\(H_i^{(s_i)}\) in \(\ol{G[A_i]}\) for each \(i\in[t]\). The parts of
orders \(m_{k+1-i}\) give the \(t\) largest parts
\(m_k,\ldots,m_{k-t+1}\) of \(H\).

Set \(r_i:=s_i-1\). The other parts in the \(i\)th copy have orders
$
 m_{k-t},m_{k-t-1},\ldots,m_{k-t-r_i+1}.
$
Moreover,
$
 \sum_{i=1}^t r_i=\sum_{i=1}^t s_i-t\ge k-t.
$
For each \(j\in[k-t]\), the number of these parts having order at least
\(m_{k-t-j+1}\) is
$
 \sum_{i=1}^t\min\{r_i,j\}
 \ge\min\left\{\sum_{i=1}^t r_i,j\right\}
 =j.
$
Hence, after arranging the available parts in decreasing order, their
orders are termwise at least
$
 m_{k-t},m_{k-t-1},\ldots,m_1.
$
By deleting vertices if necessary, we obtain the remaining \(k-t\)
parts of \(H\). Since the sets \(A_1,\ldots,A_t\) are pairwise
anticomplete in \(G\), all edges between parts chosen from distinct
sets are present in \(\ol G\). Thus \(\ol G\) contains \(H\), a
contradiction. Therefore
\begin{equation}\label{eq:sum-si}
 \sum_{i=1}^t s_i\le k-1.
\end{equation}
Since \(s_j\ge1\) for every \(j\in[t]\), we have
\(s_i+t-1\le\sum_{j=1}^t s_j\le k-1\), and hence \(s_i\le k-t\).
The maximality of \(s_i\) now implies that
\(\ol{G[A_i]}\) is \(H_i^{(s_i+1)}\)-free.

If \(s_i>1\), apply
Lemma~\ref{lem:local-master}\textup{(i)} to \(G[A_i]\), with
\(H'=H_i^{(s_i+1)}\) and \(s=s_i\). Since \(G[A_i]\) is \(C_n\)-free,
we obtain \(|A_i|<s_i n-n/10\). If \(s_i=1\), then
\(\ol{G[A_i]}\) is \(K_{m_{k-t},m_{k+1-i}}\)-free, so
Corollary~\ref{cor:PS-bip}, with \(a=m_{k-t}\) and
\(b=m_{k+1-i}\), gives \(|A_i|<n+m_{k-t}\). Here we take \(C_3\ge\max\{C_2,2\times10^{49}\}\).

Since \(S\) and the components of order less than \(m_2\) contain fewer
than \(10^{18}k+m_2\) vertices,
\begin{equation}\label{eq:sum-A-lower}
 \sum_{i=1}^t|A_i|>
 (k-1)(n-1)+z-10^{18}k-m_2.
\end{equation}

If \(s_i>1\) for some \(i\), then the preceding bounds and
\eqref{eq:sum-si} give
\(\sum_{j=1}^t|A_j|<(k-1)n-n/10+(t-1)m_{k-t}
\le(k-1)n-n/10+h\), since \((t+1)m_{k-t}\le h\). This is a
contradiction to \eqref{eq:sum-A-lower}. Hence \(s_i=1\) for every
\(i\). If \(t<k-1\), then
\(\sum_{i=1}^t|A_i|<t(n+m_{k-t})\le(k-2)n+h\), since
\(tm_{k-t}\le h\). This is again a contradiction to
\eqref{eq:sum-A-lower}. Thus \(t=k-1\).\hfill $\blacksquare$
\medskip

For every \(i\), we have \(|A_i|\ge n-|S|\ge0.95n\). Since
\(m_{k-t}=m_1\), if \(\ol{G[A_i]}\) contained \(K_{m_1,m_2}\), we could
choose parts of orders \(m_3,\ldots,m_k\), one from each of the remaining
\(k-2\) sets. As the sets \(A_1,\ldots,A_{k-1}\) are pairwise
anticomplete, these parts would form a copy of \(H\) in \(\ol G\), a
contradiction. Therefore \textup{(a)}--\textup{(c)} hold for
\(A_1,\ldots,A_{k-1}\).
\end{proof}

The next lemma extracts an almost spanning expanding subgraph from each
large set from the outcome (ii) of Lemma \ref{lem:stability}.
\begin{lemma}\label{lem:large-expander}
Let $b\ge a\ge1$, let $N=|G|\ge10^5(a+b)$, and suppose that $\ol G$ is
$K_{a,b}$-free. Set $d:=\lceil(a+b)/2\rceil$ and $\beta:=N/10^4$.
Then $G$ contains an induced subgraph $F$ with $|F|\ge N-d$ which
multi-$(d,2,2,\beta,16)$-expands. Moreover, $\beta\ge2d$ and
$|F|\le p_{2,2}(d,2\beta)$.
\end{lemma}

\begin{proof}
Write $\phi:=\frac{\eta(1+\log 2)}{\log 4}$. The assumptions give
$\beta\ge2d$ and $0.4N\ge b$.

For every $S\subseteq V(G)$ with $d\le|S|\le N/2$, we have
\begin{equation}\label{eq:medium-expansion}
 |N_G(S)|\ge\phi|S|+20\beta.
\end{equation}
Otherwise
$T:=V(G)\setminus(S\cup N_G(S))$ has at least $0.4N\ge b$ vertices.
Since $d\ge a$, an $a$-set in $S$ and a $b$-set in $T$ give a copy of
$K_{a,b}$ in $\ol G$.

If a nonempty set $X$ with $|X|\le2d$ and
$|N_G(X)|<16|X|$ exists, choose one of maximum order; otherwise set
$X=\varnothing$. Equation \eqref{eq:medium-expansion} and
$\beta\ge2d$ imply $|X|<d$. Set $F:=G-X$. If
$0<|Y|\le d$ and $|N_F(Y)|<16|Y|$, then
$|X\cup Y|<2d$ and
$|N_G(X\cup Y)|<16|X\cup Y|$, a contradiction to the maximality. For $d\le|Y|\le\beta$, equation \eqref{eq:medium-expansion}, after deleting
fewer than $d$ vertices, gives $|N_F(Y)|\ge16|Y|$.  For
$\beta\le|Y|\le|F|/2$, it gives
$|N_F(Y)|\ge\phi|Y|+10\beta$. Thus
$F$ multi-$(d,2,2,\beta,16)$-expands, and
$|F|\ge N-d$.
Finally, the definition of $p_2$ gives
$p_{2,2}(d,2\beta)=2d+\lceil3K\beta\rceil\ge N$, and hence
$|F|\le p_{2,2}(d,2\beta)$.
\end{proof}

The following lemma shows that two vertex-disjoint paths suffice to join a protected switcher to a disjoint cycle, with only bounded losses in length and flexibility and with most of the cycle forming a free subpath.
\begin{lemma}\label{lem:two-path-joining}
Let $b\ge a\ge1$. Let $F$ be a protected switcher of length $\ell_1$
and flexibility $\Omega$. Suppose that every protected block has order at most $b_0\ge1$ and flexibility at most $\omega$. Let $C$ be a cycle of length $\ell_2$, disjoint from $F$. Suppose that
$\ol{G[V(F)]}$ and $\ol{G[V(C)]}$ are $K_{a,b}$-free, and that two
vertex-disjoint paths join $F$ to $C$, each with length at most $t$. Then \(G\) contains a protected switcher $J$ satisfying
\begin{align}
 \ell(J)&\ge \ell_1+\ell_2-2(a+b)-10b_0,
 \label{eq:two-join-low}\\
 \ell(J)&\le \ell_1+\ell_2+2t+10b_0,
 \label{eq:two-join-up}\\
 \Omega(J)&\ge \Omega-(a+b)-10\omega.
 \label{eq:two-join-flex}
\end{align}
Moreover, $J$ has a free subpath of length at least
\begin{equation}\label{eq:two-join-free}
 \ell(J)-\ell_1-10b_0.
\end{equation}
\end{lemma}

\begin{proof}
Replace each joining path by the subpath from its last vertex in $F$ to its first vertex in $C$. Denote the resulting paths by $P_1,P_2$, with endvertices $x_1,x_2\in V(F)$ and $y_1,y_2\in V(C)$, respectively. Thus $P_1$ and $P_2$
are internally disjoint from $F\cup C$.

We first find an $x_1$--$x_2$ path $Q_F$ in \(G[V(F)]\). Let $R_0$ be the longest route of $F$. Suppose that $x_1$ and $x_2$ do not lie internally on opposite sides of the same switch. Choose a route $R$ containing both vertices, using in each switch met by $x_1$ or $x_2$ the side containing that vertex and using
the long side in every other switch. The route $R$ differs from $R_0$ only inside at most two protected blocks, and hence $\ell(R)\ge\ell_1-2b_0$.

If one of the two $x_1$--$x_2$ paths of $R$ has length at most $b$, let $Q_F$ be the other path. Otherwise, orient $R$ from $x_1$ towards $x_2$. Let $A$ be the $a$ vertices immediately preceding $x_1$ and let $B$ be the $b$ vertices immediately preceding $x_2$. The two sets are disjoint. Since $\ol{G[V(F)]}$ is $K_{a,b}$-free, there are $z_1\in A$ and $z_2\in B$ with $z_1z_2\in E(G)$. Follow $R$ from $x_1$ to $z_2$, use the edge $z_2z_1$, and then follow $R$ in the opposite direction from $z_1$ to $x_2$. This gives an $x_1$--$x_2$ path $Q_F$. In either case, the edges of $R$ not used by $Q_F$ form at most two paths of total length at most $a+b$. Consequently, $\ell_1- (a+b)-2b_0\le\ell(Q_F)\le\ell_1$.

For every protected block $B'$, we have
$\Omega(B')\le\ell(\operatorname{supp}(B'))$, since the support contains the long side of every switch in $B'$. Exclude each protected block in which $R$ differs from $R_0$, each protected block containing $x_1$ or $x_2$, and each
protected block whose support meets both $Q_F$ and one of the omitted paths. There are at most eight such blocks. Among the remaining blocks, also exclude those whose supports are contained in the omitted arcs. Their total flexibility
is at most $a+b$. Thus the protected blocks which remain have total flexibility at least $\Omega-(a+b)-10\omega$, and their supports are contained in $Q_F$, where $Q_F$ uses the long side of every switch in them.

Suppose now that $x_1$ and $x_2$ lie internally on opposite sides of the same switch. Let the side lengths of this switch be $L\ge S$, and let $E$ be the length of the longest route outside this switch, so $\ell_1=E+L$. The two
$x_1$--$x_2$ paths using the external part of the longest route have total length of $2E+L+S$. Choose the longer one as $Q_F$. Since $L,S\le b_0$, we have $\ell_1-b_0\le\ell(Q_F)\le\ell_1+b_0$. Exclude the protected block containing this switch. All other protected blocks remain unchanged, and their total flexibility is at least $\Omega-\omega$.

In both cases, we have an $x_1$--$x_2$ path $Q_F$ and a collection of unchanged protected blocks whose supports lie in $Q_F$, such that $Q_F$ uses the long side of every switch in these blocks, $\ell_1-(a+b)-10b_0\le\ell(Q_F)\le\ell_1+b_0$, and their total flexibility is at least $\Omega-(a+b)-10\omega$; see
Figure~\ref{fig:two-path-joining}.

Apply the same argument to $C$. Since $\ol{G[V(C)]}$ is $K_{a,b}$-free,
there is a $y_1$--$y_2$ path $Q_C\subseteq G[V(C)]$ with
$\ell_2-(a+b)\le\ell(Q_C)\le\ell_2$. The union
$Q_F\cup P_1\cup Q_C\cup P_2$ is a cycle. Add the unused sides of all switches in the protected blocks chosen above. The resulting graph is a protected switcher $J$, and its longest route is this cycle. Hence
$\ell(J)=\ell(Q_F)+\ell(P_1)+\ell(Q_C)+\ell(P_2)$. Since
$1\le\ell(P_i)\le t$, the bounds on $Q_F$ and $Q_C$ give
\eqref{eq:two-join-low} and \eqref{eq:two-join-up}, while the flexibility
estimate gives \eqref{eq:two-join-flex}.

\begin{figure}[t]
\centering

\begin{minipage}[t]{0.48\textwidth}
\centering
\begin{tikzpicture}[
    scale=0.78,
    line cap=round,
    line join=round,
    vertex/.style={
        circle,
        fill=black,
        inner sep=1.25pt
    },
    route/.style={
        line width=1.15pt,
        blue!75!black
    },
    switchside/.style={
        line width=0.9pt,
        gray!65
    },
    omitted/.style={
        line width=0.9pt,
        dashed,
        gray!65
    },
    joiningpath/.style={
        line width=1.05pt,
        red!75!black
    }
]

\node at (-5.1,1.55) {\textup{(a)}};

%---------------- F ----------------
\coordinate (x1a) at (-1.5,0.80);
\coordinate (x2a) at (-1.5,-0.80);

\coordinate (v1a) at (-4.30,0.20);
\coordinate (w1a) at (-3.65,-0.05);
\coordinate (v2a) at (-3.20,-0.42);
\coordinate (w2a) at (-2.55,-0.55);
\coordinate (v3a) at (-2.25,-0.55);
\coordinate (w3a) at (-1.82,-0.28);

% Q_F before the first switch
\draw[route]
  (x1a)
  .. controls (-2.45,1.35) and (-4.45,1.15) ..
  (v1a);

% switch 1
\draw[route]
  (v1a)
  .. controls (-4.05,-0.35) and (-3.82,-0.42) ..
  (w1a);
\draw[switchside]
  (v1a)
  .. controls (-4.02,0.48) and (-3.82,0.38) ..
  (w1a);

% connector
\draw[route] (w1a) -- (v2a);

% switch 2
\draw[route]
  (v2a)
  .. controls (-3.02,-0.92) and (-2.75,-0.98) ..
  (w2a);
\draw[switchside]
  (v2a)
  .. controls (-3.00,-0.12) and (-2.73,-0.18) ..
  (w2a);

% connector
\draw[route] (w2a) -- (v3a);

% switch 3
\draw[route]
  (v3a)
  .. controls (-2.10,-0.95) and (-1.92,-0.82) ..
  (w3a);
\draw[switchside]
  (v3a)
  .. controls (-2.08,-0.25) and (-1.94,-0.12) ..
  (w3a);

% end of Q_F
\draw[route]
  (w3a)
  .. controls (-1.58,-0.35) and (-1.50,-0.55) ..
  (x2a);

% omitted short x1--x2 subpath
\draw[omitted]
  (x1a)
  .. controls (-0.92,0.35) and (-0.92,-0.35) ..
  (x2a);

\foreach \p in {v1a,w1a,v2a,w2a,v3a,w3a}{
  \node[vertex] at (\p) {};
}

\node[vertex] at (x1a) {};
\node[vertex] at (x2a) {};
\node[above right] at (x1a) {$x_1$};
\node[below right] at (x2a) {$x_2$};

\node at (-3.10,1.45) {$F$};
\node[text=blue!75!black] at (-3.05,-1.20) {$Q_F$};

%---------------- C ----------------
\coordinate (y1a) at (1.75,0.80);
\coordinate (y2a) at (1.75,-0.80);
\coordinate (za)  at (4.45,0);

% long path Q_C
\draw[route]
  (y1a)
  .. controls (2.75,1.55) and (3.85,1.30) ..
  (za)
  .. controls (3.85,-1.30) and (2.75,-1.55) ..
  (y2a);

% omitted short subpath of C
\draw[omitted]
  (y1a)
  .. controls (1.15,0.35) and (1.15,-0.35) ..
  (y2a);

\node[vertex] at (y1a) {};
\node[vertex] at (y2a) {};
\node[above left] at (y1a) {$y_1$};
\node[below left] at (y2a) {$y_2$};

\node at (3.45,1.43) {$C$};
\node[text=blue!75!black] at (3.55,-1.18) {$Q_C$};

%---------------- P1, P2 ----------------
\draw[joiningpath]
  (x1a)
  .. controls (-0.50,1.10) and (0.70,1.10) ..
  (y1a)
  node[midway,above=3pt,text=red!75!black] {$P_1$};

\draw[joiningpath]
  (x2a)
  .. controls (-0.50,-1.10) and (0.70,-1.10) ..
  (y2a)
  node[midway,below=3pt,text=red!75!black] {$P_2$};

\end{tikzpicture}
\end{minipage}
\hfill
\begin{minipage}[t]{0.48\textwidth}
\centering
\begin{tikzpicture}[
    scale=0.78,
    line cap=round,
    line join=round,
    vertex/.style={
        circle,
        fill=black,
        inner sep=1.25pt
    },
    route/.style={
        line width=1.15pt,
        blue!75!black
    },
    switchside/.style={
        line width=0.9pt,
        gray!65
    },
    omitted/.style={
        line width=0.9pt,
        dashed,
        gray!65
    },
    joiningpath/.style={
        line width=1.05pt,
        red!75!black
    }
]

\node at (-5.1,1.55) {\textup{(b)}};

%---------------- F ----------------

% distinguished vertices of the special switch
\coordinate (sa) at (-2.35,0);
\coordinate (ta) at (-0.45,0);

% x1 and x2 lie internally on its opposite sides
\coordinate (x1b) at (-1.40,0.63);
\coordinate (x2b) at (-1.40,-0.63);

% Q_F uses the left part of the upper side
\draw[route]
  (sa)
  .. controls (-2.05,0.52) and (-1.72,0.67) ..
  (x1b);

% unused remainder of the upper side
\draw[switchside]
  (x1b)
  .. controls (-1.08,0.67) and (-0.72,0.52) ..
  (ta);

% unused left part of the lower side
\draw[switchside]
  (sa)
  .. controls (-2.05,-0.52) and (-1.72,-0.67) ..
  (x2b);

% Q_F uses the right part of the lower side
\draw[route]
  (x2b)
  .. controls (-1.08,-0.67) and (-0.72,-0.52) ..
  (ta);

% external part of Q_F, containing two further switches
\coordinate (v1b) at (-4.30,0.10);
\coordinate (w1b) at (-3.65,-0.15);
\coordinate (v2b) at (-3.15,-0.55);
\coordinate (w2b) at (-2.50,-0.55);

\draw[route]
  (sa)
  .. controls (-3.05,0.85) and (-4.05,0.70) ..
  (v1b);

% switch 1
\draw[route]
  (v1b)
  .. controls (-4.05,-0.42) and (-3.82,-0.48) ..
  (w1b);
\draw[switchside]
  (v1b)
  .. controls (-4.02,0.40) and (-3.82,0.32) ..
  (w1b);

\draw[route] (w1b) -- (v2b);

% switch 2
\draw[route]
  (v2b)
  .. controls (-2.98,-1.00) and (-2.70,-1.00) ..
  (w2b);
\draw[switchside]
  (v2b)
  .. controls (-2.98,-0.22) and (-2.70,-0.22) ..
  (w2b);

\draw[route]
  (w2b)
  .. controls (-1.75,-1.15) and (-0.72,-0.72) ..
  (ta);

\foreach \p in {sa,ta,v1b,w1b,v2b,w2b}{
  \node[vertex] at (\p) {};
}

\node[vertex] at (x1b) {};
\node[vertex] at (x2b) {};

\node[above] at (x1b) {$x_1$};
\node[below] at (x2b) {$x_2$};

\node at (-3.15,1.45) {$F$};
\node[text=blue!75!black] at (-3.30,-1.22) {$Q_F$};

%---------------- C ----------------
\coordinate (y1b) at (1.75,0.80);
\coordinate (y2b) at (1.75,-0.80);
\coordinate (zb)  at (4.45,0);

\draw[route]
  (y1b)
  .. controls (2.75,1.55) and (3.85,1.30) ..
  (zb)
  .. controls (3.85,-1.30) and (2.75,-1.55) ..
  (y2b);

\draw[omitted]
  (y1b)
  .. controls (1.15,0.35) and (1.15,-0.35) ..
  (y2b);

\node[vertex] at (y1b) {};
\node[vertex] at (y2b) {};

\node[above left] at (y1b) {$y_1$};
\node[below left] at (y2b) {$y_2$};

\node at (3.45,1.43) {$C$};
\node[text=blue!75!black] at (3.55,-1.18) {$Q_C$};

%---------------- P1, P2 ----------------
\draw[joiningpath]
  (x1b)
  .. controls (-0.45,1.05) and (0.70,1.08) ..
  (y1b)
  node[midway,above=3pt,text=red!75!black] {$P_1$};

\draw[joiningpath]
  (x2b)
  .. controls (-0.45,-1.05) and (0.70,-1.08) ..
  (y2b)
  node[midway,below=3pt,text=red!75!black] {$P_2$};

\end{tikzpicture}
\end{minipage}

\caption{
The two cases in the construction of \(Q_F\) in
Lemma~\ref{lem:two-path-joining}.
In \textup{(a)}, \(x_1\) and \(x_2\) do not lie internally on opposite
sides of the same switch. In \textup{(b)}, they lie internally on
opposite sides of the same switch. In both cases, \(Q_F\) and \(Q_C\)
are joined by \(P_1\) and \(P_2\) to form the longest route of \(J\).
}
\label{fig:two-path-joining}
\end{figure}
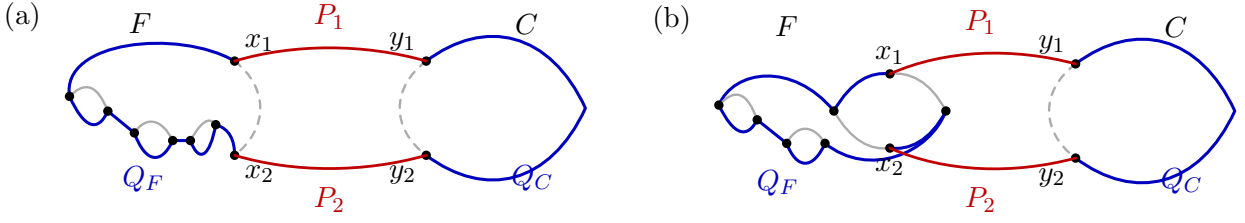

The $x_1$--$x_2$ subpath of the longest route through $P_1$, $Q_C$, and $P_2$
has length $\ell(J)-\ell(Q_F)\ge\ell(J)-\ell_1-b_0$. Delete its first and
last edges. The resulting subpath lies outside $F$ and is contained in every
route of $J$, so it is free. Its length is at least
$\ell(J)-\ell_1-b_0-2\ge\ell(J)-\ell_1-10b_0$, proving
\eqref{eq:two-join-free}.
\end{proof}

We now prove the multipartite form of the main theorem.

\begin{theorem}\label{thm:multipartite}
There is an absolute constant $C_5$ such that the following holds.
Let $H=K_{m_1,\ldots,m_k}$ have order $h$, set $m=h/k$, and suppose that
$m_1\le\cdots\le m_k$, $k\ge2$, and
$2(k-1)/k\le m\le k^{22}$. If $n\ge C_5h$, then
$R(C_n,H)=(k-1)(n-1)+m_1$.
\end{theorem}

\begin{proof}
Take \(C_5\) sufficiently large. The lower bound follows from
\eqref{eq:burr-intro}. We prove the upper bound by induction on \(k\).
For \(k<k_0\), it follows from Theorem~\ref{thm:HHKL}, since \(k_0\) is
absolute. Let \(k\ge k_0\), and let \(G\) be a \(C_n\)-free graph on
\((k-1)(n-1)+m_1\) vertices such that \(\ol G\) is \(H\)-free. Apply
Lemma~\ref{lem:stability} with \(z=m_1\).

Suppose that Lemma~\ref{lem:stability}\textup{(i)} holds. Since
\((k-1)m_2\le h\), the average part order of \(\widehat H\) satisfies
$\widehat m:=
 \frac{|\widehat H|}{k-1}
 =\frac{h-m_2}{k-1}
 \ge m\frac{k(k-2)}{(k-1)^2}
 \ge2\frac{k-2}{k-1}.
$
Also \(\sigma(\widehat H)=m_1\). If \(\widehat m\le(k-1)^{22}\), then
\(2(k-2)/(k-1)\le\widehat m\le(k-1)^{22}\) and
\(n\ge C_5|\widehat H|\). Hence the induction hypothesis gives
$
 R(C_n,\widehat H)=(k-2)(n-1)+m_1.
$
Since Lemma~\ref{lem:stability}\textup{(i)} gives an induced subgraph
\(G'\subseteq G\) of order at least \((k-2)(n-1)+m_1\) with
\(\ol{G'}\) being \(\widehat H\)-free, the graph \(G'\) contains
\(C_n\), a contradiction. Otherwise \(\widehat m>(k-1)^{22}\). Hence the largest part of \(\widehat H\) has order greater than \((k-1)^{22}\). Since \(n\ge C_5h\ge10^{60}m_{\max}(\widehat H)\), Corollary~\ref{cor:PS-large} implies that \(G'\) contains \(C_n\), a contradiction. Thus Lemma~\ref{lem:stability}\textup{(ii)} holds. Let \(A_1,\ldots,A_{k-1}\) be the resulting sets.

For each \(i\), apply Lemma~\ref{lem:large-expander} to \(G[A_i]\)
with \(a=m_1\), \(b=m_2\), \(N=|A_i|\), \(d:=\lceil(m_1+m_2)/2\rceil\), and \(\beta_i:=|A_i|/10^4\). We have \(m_2\ge m_1\), and
\(\ol{G[A_i]}\) is \(K_{m_1,m_2}\)-free by Lemma~\ref{lem:stability}\textup{(ii)(b)}. Moreover,
\(m_1+m_2\le2m\le h\), so, by increasing \(C_5\),
$
 |A_i|\ge0.95n\ge10^5(m_1+m_2).
$
Hence there is an induced subgraph \(X_i\subseteq G[A_i]\) which
multi-\((d,2,2,\beta_i,16)\)-expands, with
\(\beta_i\ge2d\), \(|X_i|\le p_{2,2}(d,2\beta_i)\), and
$
 |X_i|\ge|A_i|-d\ge|A_i|-m-1.
$
Each \(X_i\) is \(2\)-connected. Indeed, if \(S\) is a smallest component of \(X_i\), or a smallest component of \(X_i-v\) for a cutvertex \(v\), then \(|S|\le|X_i|/2\) and \(|N_{X_i}(S)|\le1\), a contradiction to the expansion inequality. We divide the proof into two cases.

\medskip
\noindent\textbf{Case 1. There exist two expanding subgraphs joined by two vertex-disjoint paths}

\smallskip

Suppose that \(X_i\) and \(X_j\) are joined by two vertex-disjoint paths.
Replace each by the subpath from its last vertex in \(X_i\) to its first
vertex in \(X_j\), with endvertices \(x_1,x_2\in X_i\) and
\(y_1,y_2\in X_j\).
Apply Lemma~\ref{lem:local-master}\textup{(ii)} to \(G[A_i]\), with
\(H'=K_{m_1,m_2}\) and \(s=1\). Its hypotheses follow from
\(|A_i|\ge0.95n\) and Lemma~\ref{lem:stability}\textup{(ii)(c)}.
We obtain a protected switcher \(F_i\) such that
\(0.64n<\ell(F_i)<0.70n\), \(20h\le\Omega(F_i)\le C_0h\), and every
protected block has order and flexibility at most \(C_0h/M\). The same
lemma applied to \(G[A_j]\), followed by taking a longest route, gives a
cycle \(C_j\) with \(0.64n<|C_j|<0.70n\).

Since \(|V(F_i)|\ge\ell(F_i)\) and
\(|X_i|\ge|A_i|-m-1\), we have
$
 |V(F_i)\cap V(X_i)|>0.64n-m-1>0.5n.
$
Similarly, \(|V(C_j)\cap V(X_j)|>0.5n\).
By Corollary~\ref{cor:PS-bip}, we have
\(|A_i|,|A_j|<n+m_1\).

We connect \(F_i\) to \(x_1,x_2\) inside \(X_i\). Set \(R_i:=\lfloor\beta_i/4\rfloor\). By Lemma~\ref{lem:multi-root}, we obtain disjoint \((x_1,R_i)\)- and \((x_2,R_i)\)-rooted sets
\(B_1,B_2\subseteq X_i\). Let
\(Z_0:=(V(F_i)\cap V(X_i))\setminus(B_1\cup B_2)\).
Then \(|Z_0|>0.49n\). By Lemma~\ref{lem:short-path} in \(X_i\) with
\(A=Z_0\), \(B=B_1\cup B_2\), and \(Z=\varnothing\), there is a
\(Z_0\)--\((B_1\cup B_2)\) path of length at most
\(E(2,2,\beta_i)\). Take a shortest \(Z_0\)--\((B_1\cup B_2)\) path, and relabel \(B_1,B_2\) so that its end in \(B_1\cup B_2\) lies in \(B_1\).
Extend it inside \(B_1\) to \(x_1\), and take the subpath from the
vertex of \(F_i\) nearest \(x_1\) to \(x_1\). Replacing this subpath by
a shortest path in its vertex set, Lemma~\ref{lem:shortcut}\textup{(i)}
gives an \(F_i\)--\(x_1\) path \(P_1\) of length at most
\(m_1+m_2\).

Set \(Z_1:=(V(F_i)\cap V(X_i))\setminus(V(P_1)\cup B_2)\). Then
\(|Z_1|>0.48n\). Apply Lemma~\ref{lem:short-path} in \(X_i\) with
\(A=Z_1\), \(B=B_2\), and \(Z=V(P_1)\). Indeed,
\(|V(P_1)|\le m_1+m_2+1\le R_i=|B_2|\le14\min\{|Z_1|,|B_2|\}\) and
\(|V(P_1)|\le5\beta_i\). Let \(Q\) be a shortest
\(Z_1\)--\(B_2\) path in \(X_i-V(P_1)\), and extend it inside \(B_2\)
to \(x_2\). Take the subpath from its last vertex in \(F_i\) to \(x_2\)
and replace it by a shortest path in its vertex set.
Lemma~\ref{lem:shortcut}\textup{(i)} gives an \(F_i\)--\(x_2\) path
\(P_2\), disjoint from \(P_1\), of length at most \(m_1+m_2\).

The same argument in \(X_j\) gives two disjoint paths from \(y_1,y_2\)
to \(C_j\). Together with \(P_1,P_2\) and the two
\(X_i\)--\(X_j\) paths, they form two vertex-disjoint
\(F_i\)--\(C_j\) paths. From each, take a minimal subpath joining
\(F_i\) to \(C_j\), and replace it by a shortest path in its vertex
set. Since \(\ol G\) is \(H\)-free,
Lemma~\ref{lem:shortcut}\textup{(i)}, applied with \(J=H\), bounds
their lengths by \(h+k\).

By the choice of \(k_0\), we have \(C_0h/M\ge1\). Set
\(b_0=\omega_0:=C_0h/M\). The hypotheses of
Lemma~\ref{lem:two-path-joining} hold with
\(a=m_1\), \(b=m_2\), and \(t=h+k\). Since
\(m_1+m_2\le h\), \(k\le h\), and
\(10C_0/M\le5\times10^{-15}\), that lemma gives
$
 \ell(J)>1.28n-(2+5\times10^{-15})h>1.19n,
$
$
 \ell(J)<1.40n+(4+5\times10^{-15})h<1.41n,
$
and
$
 \Omega(J)\ge(19-5\times10^{-15})h>10h,
$
after increasing \(C_5\). It also gives a free subpath \(Q\) with
$
 \ell(Q)\ge\ell(J)-\ell(F_i)-10b_0
 >\ell(J)-0.70n-10b_0
 \ge\ell(J)-n+h+k-2.
$
Also
\(\Omega(J)\le\Omega(F_i)\le C_0h\).

Set \(d_0:=h+k-2\), and let \(L\) and \(S\) denote the current
longest-route length and the length of the free subpath \(Q\),
respectively. By above, \(S\ge L-n+d_0\). While \(L>n+d_0\), this
inequality gives \(S>2d_0\), so \(Q\) contains a subpath of \(d_0\)
edges. Every route has length at least
\(L-\Omega(J)\ge n-C_0h\ge h+k\). Hence
Lemma~\ref{lem:shortcut}\textup{(ii)}, applied with \(J=H\), replaces
this subpath and shortens every route by an integer
\(q\in[1,d_0-1]\). The free subpath also loses \(q\) edges, so
\(S\ge L-n+d_0\) remains valid.
The process terminates with \(n<L\le n+d_0\), since the final
replacement starts with \(L>n+d_0\) and shortens it by at most
\(d_0-1\). As \(\Omega(J)>10h>d_0\), we have
\(L-\Omega(J)<L-d_0\le n<L\). By Lemma~\ref{lem:interval-routes}, we obtain a route of length \(n\), a contradiction.

\medskip
\noindent\textbf{Case 2. No two subgraphs \(X_i\) and \(X_j\) are joined by two vertex-disjoint paths.}

\smallskip

\begin{claim}
There are \(i\in[k-1]\) and \(v\in V(G)\) such that the component of
\(G-v\) containing \(X_i-v\) contains no \(X_j-v\) with \(j\ne i\).
\end{claim}

If some component of \(G\) contains only one \(X_i\), any
\(v\in V(X_i)\) has the required property, since \(X_i-v\) is connected.

Otherwise, fix a component \(D\) containing at least two of the
subgraphs \(X_i\). For each \(X_j\subseteq D\), add a vertex \(z_j\)
adjacent to every vertex of \(X_j\), and let \(D_j\) be the block
containing \(X_j\cup\{z_j\}\). The blocks \(D_j\) are distinct, since
two internally vertex-disjoint \(z_j\)--\(z_{j'}\) paths would give two
vertex-disjoint \(X_j\)--\(X_{j'}\) paths.
Form the incidence tree of the blocks and cutvertices of the enlarged
component, and take the minimal subtree containing all \(D_j\). Let
\(D_i\) be a leaf of this subtree, and let \(v\) be its unique adjacent
cutvertex. Since no \(z_j\) is a cutvertex, \(v\in V(G)\). Every path
from \(X_i-v\) to \(X_j-v\), \(j\ne i\), passes through \(v\), proving
the claim. \hfill\(\blacksquare\)

\medskip

Let \(A\) be the component of \(G-v\) containing \(X_i-v\), and set
\(G':=G[A\cup\{v\}]\). The graph \(\ol{G'}\) is
\(K_{m_1+1,m_2+1}\)-free. Indeed, such a copy contains a
\(K_{m_1,m_2}\) in \(A\). Choosing sets of orders \(m_3,\ldots,m_k\),
one from each \(X_j\setminus\{v\}\), \(j\ne i\), then gives a copy of
\(H\) in \(\ol G\), a contradiction.

Let \(B\subseteq V(G')\) have order at least \(m_2+1\), and choose an
\(m_2\)-set \(B_0\subseteq B\setminus\{v\}\). Since \(A\) is a component
of \(G-v\), \(N_G(B_0)\subseteq A\cup\{v\}\). If
\(|B\cup N_{G'}(B)|\le n-1\), then
\(|B_0\cup N_G(B_0)|\le n-1\), so
\(U:=V(G)\setminus(B_0\cup N_G(B_0))\) has order at least
\((k-2)(n-1)+m_1\). As Lemma~\ref{lem:stability}\textup{(i)} does not
hold, \(\ol{G[U]}\) contains \(\widehat H\). Together with \(B_0\), this
forms a copy of \(H\) in \(\ol G\), a contradiction. Hence
\(|B\cup N_{G'}(B)|\ge n\).

Since \(X_i\subseteq A_i\), \(\ol{G[A_i]}\) is \(K_{m_1,m_2}\)-free,
and \(m_1\le m_2\), the graph \(\ol{X_i}\) is
\(K_{m_2,m_2}\)-free. Moreover, \(2m_2\le h\) and
\(|X_i|\ge|A_i|-m-1\ge n/10\). Lemma~\ref{lem:long-cycle}, applied to
\(X_i\) with \(L=K_{m_2,m_2}\), gives a cycle \(C\) of length at least
\(n/10^{10}\ge8\max\{m_2+1,8\}\). Choose \(xy\in E(C)\). Then \(C-xy\)
is an \(x\)--\(y\) path on at least \(8\max\{m_2+1,8\}\) vertices.
Since \(K_{m_2+1,m_2+1}\) contains \(K_{m_1+1,m_2+1}\), the graph
\(\ol{G'}\) is \(K_{m_2+1,m_2+1}\)-free. Applying
Lemma~\ref{lem:PS-path} with parameter \(m_2+1\) gives an \(x\)--\(y\)
path on exactly \(n\) vertices. Together with \(xy\), it forms \(C_n\),
a contradiction.
\end{proof}

\subsection{Proof of the main theorem}
Recall that \(k_0=k_0(M)\) is the sufficiently large absolute constant fixed above.
\begin{proof}[Proof of Theorem~\ref{thm:main}]

Choose the final absolute constant \(C\) sufficiently large, with \(C\ge3\).
Set \(k:=\chi(H)\) and \(h:=|H|\). If \(k=1\), then \(e(H)=0\). Hence every
graph on \(h\) vertices contains \(H\) in its complement. Together with
\eqref{eq:burr-intro}, this gives \(R(C_n,H)=h=\sigma(H)\).

Assume that \(k\ge2\). Choose a proper \(k\)-coloring of \(H\) whose smallest
color class has order \(\sigma(H)\), and relabel the color classes so that
their orders satisfy \(m_1\le\cdots\le m_k\). Adding all missing edges between
distinct color classes gives a complete \(k\)-partite supergraph
\(H'=K_{m_1,\ldots,m_k}\) with \(|H'|=h\) and \(m_1=\sigma(H)\).
It suffices to prove the upper bound for \(H'\).

If \(k=2\), the result follows from Corollary~\ref{cor:PS-bip}. If
\(3\le k<k_0\), the result follows from Theorem~\ref{thm:HHKL}, after
increasing \(C\). Hence assume that \(k\ge k_0\), and set \(m:=h/k\).

If \(m>k^{22}\), then \(m_k\ge m>k^{22}\), and the result follows from
Corollary~\ref{cor:PS-large}. Hence assume that \(m\le k^{22}\).
If \(h<2(k-1)\), enlarge one of the largest parts of \(H'\) by
\(2(k-1)-h\) vertices and denote the resulting graph by \(H''\);
otherwise, let \(H''=H'\). Then \(H'\subseteq H''\),
\(\sigma(H'')=\sigma(H)\), and \(|H''|=\max\{h,2(k-1)\}\le2h\).
Moreover, the average part order of \(H''\) lies between
\(2(k-1)/k\) and \(k^{22}\). Thus Theorem~\ref{thm:multipartite},
applied to \(H''\), gives the required upper bound after increasing
\(C\) by a factor of \(2\).

Thus every \(C_n\)-free graph on \((k-1)(n-1)+\sigma(H)\) vertices contains
\(H\) in its complement. Hence
\(R(C_n,H)\le(k-1)(n-1)+\sigma(H)\). The reverse inequality follows from
\eqref{eq:burr-intro}.
\end{proof}

\section{Concluding remarks}
Theorem~\ref{thm:main} also has the following equivalent cycle-spectrum formulation. Let $H$ be a graph with $\chi(H)=k\ge 2$, and let $G$ be an $N$-vertex graph such that $\overline G$ is $H$-free. Then $G$ contains a cycle of every length $\ell$ satisfying
\[
C|H|\le \ell\le
\left\lfloor\frac{N-\sigma(H)}{k-1}\right\rfloor+1.
\]

Our main theorem also gives a quantitative bound of the multicolor version of Ramsey goodness proved by Allen, Brightwell and Skokan~\cite{ABS}. We first recall the definitions in~\cite{ABS}. Let $H_1,\ldots,H_r$ be fixed nonempty graphs. The \textit{homomorphism Ramsey number} $R_{hom}(H_1,\ldots, H_r)$ is the smallest $N$ such that, if $F$ is any $r$-colored complete graph on $N$ vertices, there exists a color $i$ such that there is a graph homomorphism from $H_i$ into the $i$th color subgraph of $F$. Let $W=R_{\mathrm{hom}}(H_1,\ldots,H_r)-1$. We define $Z$ as the smallest natural number $N$ with the following property. For any labelled graph $F$ on $W + N$ vertices, with all edges incident to at least one of the first $W$ vertices present, if $F$ is $r$-colored and $F'$ is obtained from $F$ as the $m$-blow-up of the first $W$ vertices for some sufficiently large $m = m(H_1,\ldots, H_r)$, then there is some $i$ such that $H_i$ is contained within the $i$th color subgraph of $F'$.
With these definitions, the proof of Theorem~8 in~\cite{ABS} yields an integer $\ell\ge Z$ such that every $r$-edge-coloring of $J=K_{\ell,\ldots,\ell,Z}$, where there are $W$ parts of order $\ell$ and one part of order $Z$, contains a color-$i$ copy of $H_i$ for some $i\in[r]$. Since $\chi(J)=W+1$, $\sigma(J)=Z$, and $|J|=W\ell+Z$, Theorem~\ref{thm:main}, together with the argument in~\cite{ABS}, implies that
\[
R(C_n,H_1,\ldots,H_r)=W(n-1)+Z
\]
whenever $n\ge C|J|=C(W\ell+Z)$.

No analogous statement can hold for arbitrary sparse graphs in the first coordinate. A construction of Brandt~\cite{Brandt}, as discussed in~\cite{ABS}, shows that, for every constant $c>0$, there is a constant $\Delta=\Delta(c)$ and arbitrarily large graphs $F$ with $\Delta(F)\le \Delta$ such that $R(F,K_3)>c|F|$.
Thus sparsity, or even bounded maximum degree, does not by itself imply Ramsey goodness; some additional structural restriction on $F$ is necessary.

Long subdivisions form a natural class of sparse graphs for which a
quantitative extension may nevertheless hold. Burr~\cite{Burr} proved
that, for every fixed connected graph $Q$ and every graph $H$, every
sufficiently large subdivision of $Q$ is $H$-good. It would be
interesting to make the dependence on $H$ explicit. More precisely,
one may ask whether, for every fixed connected graph $Q$, there exists
a constant $C_Q$ such that every subdivision $F$ of $Q$ is $H$-good
whenever $|F|\ge C_Q|H|$. Theorem~\ref{thm:main} answers this question
when $Q$ is a cycle.

\section*{Declaration on the use of AI} 
The main ideas and proof strategies in this paper were developed by the authors. ChatGPT 5.6 was used to assist with optimizing numerical constants, as well as with improving the language of the manuscript. All AI-assisted suggestions were independently checked and revised by the authors, who take full responsibility for the mathematical content of the paper.


\begin{thebibliography}{99}

\bibitem{ABS}
P. Allen, G. Brightwell, J. Skokan,
Ramsey-goodness---and otherwise,
\emph{Combinatorica} \textbf{33} (2013), 125--160.

%\bibitem{Alon2025}
%N. Alon, M. Buci\'c, L. Sauermann, D. Zakharov, O. Zamir,
%Essentially tight bounds for rainbow cycles in proper edge-colourings,
%\emph{Proc. Lond. Math. Soc.} \textbf{130} (2025), no.~4,
%Paper No.~e70044.

\bibitem{BBCGHMST}
P. Balister, B. Bollob\'as, M. Campos, S. Griffiths, E. Hurley,
R. Morris, J. Sahasrabudhe, M. Tiba,
Upper bounds for multicolour Ramsey numbers,
\emph{J. Amer. Math. Soc.} \textbf{39} (2026), 765--780.

\bibitem{BPS}
I. Balla, A. Pokrovskiy, B. Sudakov,
Ramsey goodness of bounded degree trees,
\emph{Combin. Probab. Comput.} \textbf{27} (2018), 289--309.

\bibitem{BondyErdos}
J. A. Bondy, P. Erd\H{o}s,
Ramsey numbers for cycles in graphs,
\emph{J. Combin. Theory Ser. B} \textbf{14} (1973), 46--54.

\bibitem{Bradac}
D. Brada\v{c},
Nearly tight exponents for off-diagonal Ramsey numbers,
arXiv:2605.28793, 2026.

\bibitem{Brandt}
S. Brandt,
Expanding graphs and Ramsey numbers,
Preprint No.~A 96-24, Freie Universit\"at Berlin (1996).

\bibitem{Burr}
S. A. Burr,
Ramsey numbers involving graphs with long suspended paths,
\emph{J. Lond. Math. Soc.} \textbf{24} (1981), 405--413.

\bibitem{BurrErdosMagnitude}
S. A. Burr, P. Erd\H{o}s,
On the magnitude of generalized Ramsey numbers for graphs,
in \emph{Infinite and Finite Sets, Vol.~I},
Colloq. Math. Soc. J\'anos Bolyai \textbf{10},
North-Holland, Amsterdam, 1975, 215--240.

\bibitem{BurrErdos}
S. A. Burr, P. Erd\H{o}s,
Generalizations of a Ramsey-theoretic result of Chv\'atal,
\emph{J. Graph Theory} \textbf{7} (1983), 39--51.

\bibitem{CGMS}
M. Campos, S. Griffiths, R. Morris, J. Sahasrabudhe,
An exponential improvement for diagonal Ramsey,
\emph{Ann. of Math. (2)} \textbf{203} (2026), 869--932.

\bibitem{Chvatal}
V. Chv\'atal,
Tree-complete graph Ramsey numbers,
\emph{J. Graph Theory} \textbf{1} (1977), 93.

\bibitem{CH}
V. Chv\'atal, F. Harary,
Generalized Ramsey theory for graphs. III. Small off-diagonal numbers,
\emph{Pacific J. Math.} \textbf{41} (1972), 335--345.

\bibitem{CRST}
V. Chv\'atal, V. R\"odl, E. Szemer\'edi, W. T. Trotter, Jr.,
The Ramsey number of a graph with bounded maximum degree,
\emph{J. Combin. Theory Ser. B} \textbf{34} (1983), 239--243.

\bibitem{CFS}
D. Conlon, J. Fox, B. Sudakov,
Recent developments in graph Ramsey theory,
in A. Czumaj, A. Georgakopoulos, D. Kr\'al, V. Lozin, O. Pikhurko
(eds.),
\emph{Surveys in Combinatorics 2015},
London Math. Soc. Lecture Note Ser. \textbf{424},
Cambridge University Press, Cambridge, 2015, 49--118.

\bibitem{CFLS}
D. Conlon, J. Fox, C. Lee, B. Sudakov,
Ramsey numbers of cubes versus cliques,
\emph{Combinatorica} \textbf{36} (2016), 37--70.

\bibitem{Erdos1947}
P. Erd\H{o}s,
Some remarks on the theory of graphs,
\emph{Bull. Amer. Math. Soc.} \textbf{53} (1947), 292--294.

\bibitem{EFRS}
P. Erd\H{o}s, R. J. Faudree, C. C. Rousseau, R. H. Schelp,
On cycle--complete graph Ramsey numbers,
\emph{J. Graph Theory} \textbf{2} (1978), 53--64.

\bibitem{EFRS85}
P. Erd\H{o}s, R. J. Faudree, C. C. Rousseau, R. H. Schelp,
Multipartite graph--sparse graph Ramsey numbers,
\emph{Combinatorica} \textbf{5} (1985), 311--318.

\bibitem{FaudreeSchelp}
R. J. Faudree, R. H. Schelp,
All Ramsey numbers for cycles in graphs,
\emph{Discrete Math.} \textbf{8} (1974), 313--329.

\bibitem{FPGMSS}
G. Fiz Pontiveros, S. Griffiths, R. Morris, D. Saxton, J. Skokan,
The Ramsey number of the clique and the hypercube,
\emph{J. Lond. Math. Soc.} \textbf{89} (2014), 680--702.

\bibitem{FHW}
J. Fox, X. He, Y. Wigderson,
Ramsey goodness of books revisited,
\emph{Adv. Comb.} (2023), Paper No.~4, 21 pp.

\bibitem{FriedmanKrivelevich}
L. Friedman, M. Krivelevich,
Cycle lengths in expanding graphs,
\emph{Combinatorica} \textbf{41} (2021), 53--74.

\bibitem{GG}
L. Gerencs\'er, A. Gy\'arf\'as,
On Ramsey-type problems,
\emph{Ann. Univ. Sci. Budapest. E\"otv\"os Sect. Math.}
\textbf{10} (1967), 167--170.

%\bibitem{GFKKL}
%I. Gil Fern\'andez, J. Kim, Y. Kim, H. Liu,
%Nested cycles with no geometric crossings,
%\emph{Proc. Amer. Math. Soc. Ser. B} \textbf{9} (2022), 22--32.

%\bibitem{HKL}
%J. Haslegrave, J. Kim, H. Liu,
%Extremal density for sparse minors and subdivisions,
%\emph{Int. Math. Res. Not. IMRN} (2022), 15505--15548.

\bibitem{HHKL}
J. Haslegrave, J. Hyde, J. Kim, H. Liu,
Ramsey numbers of cycles versus general graphs,
\emph{Forum Math. Sigma} \textbf{11} (2023),
Paper No.~e10, 18 pp.

\bibitem{KLS}
P. Keevash, E. Long, J. Skokan,
Cycle-complete Ramsey numbers,
\emph{Int. Math. Res. Not. IMRN} (2021), 275--300.

%\bibitem{KimLiuSharifzadehStaden}
%J. Kim, H. Liu, M. Sharifzadeh, K. Staden,
%Proof of Koml\'os's conjecture on Hamiltonian subsets,
%\emph{Proc. Lond. Math. Soc.} \textbf{115} (2017), 974--1013.

\bibitem{KS1}
J. Koml\'os, E. Szemer\'edi,
Topological cliques in graphs,
\emph{Combin. Probab. Comput.} \textbf{3} (1994), 247--256.

\bibitem{KS2}
J. Koml\'os, E. Szemer\'edi,
Topological cliques in graphs II,
\emph{Combin. Probab. Comput.} \textbf{5} (1996), 79--90.

\bibitem{KrivelevichLocallySparse}
M. Krivelevich,
Finding and using expanders in locally sparse graphs,
\emph{SIAM J. Discrete Math.} \textbf{32} (2018), 611--623.

\bibitem{KrivelevichExp}
M. Krivelevich,
Expanders---how to find them, and what to find in them,
in A. Lo, R. Mycroft, G. Perarnau, A. Treglown (eds.),
\emph{Surveys in Combinatorics 2019},
London Math. Soc. Lecture Note Ser. \textbf{456},
Cambridge University Press, Cambridge, 2019, 115--142.

\bibitem{KrivelevichLongCycles}
M. Krivelevich,
Long cycles in locally expanding graphs, with applications,
\emph{Combinatorica} \textbf{39} (2019), 135--151.

\bibitem{Lee}
C. Lee,
Ramsey numbers of degenerate graphs,
\emph{Ann. of Math. (2)} \textbf{185} (2017), 791--829.

\bibitem{Letzter}
S. Letzter,
Sublinear expanders and their applications,
in F. Fischer, R. Johnson (eds.),
\emph{Surveys in Combinatorics 2024},
London Math. Soc. Lecture Note Ser. \textbf{493},
Cambridge University Press, Cambridge, 2024, 89--130.

\bibitem{LMMader}
H. Liu, R. Montgomery,
A proof of Mader's conjecture on large clique subdivisions in \(C_4\)-free graphs,
\emph{J. Lond. Math. Soc.} \textbf{95} (2017), 203--222.

\bibitem{LMOdd}
H. Liu, R. Montgomery,
A solution to Erd\H{o}s and Hajnal's odd cycle problem,
\emph{J. Amer. Math. Soc.} \textbf{36} (2023), 1191--1234.

\bibitem{MSX}
J. Ma, W. Shen, S. Xie,
An exponential improvement for Ramsey lower bounds,
\emph{Invent. Math.} (2026),
doi:10.1007/s00222-026-01421-9.

\bibitem{MontgomeryMinors}
R. Montgomery,
Logarithmically-small minors and topological minors,
\emph{J. Lond. Math. Soc.} \textbf{91} (2015), 71--88.

\bibitem{MontgomeryICM}
R. Montgomery,
Recent progress in graph theory using expansion,
in S. Friedlander, Y. Tschinkel (eds.),
\emph{Proceedings of the International Congress of Mathematicians
2026, Vol.~6: Invited Lectures (Sections 12--14)},
SIAM, 2026, 179--198.

\bibitem{MPY}
R. Montgomery, M. Pavez-Sign\'e, J. Yan,
Ramsey numbers of bounded degree trees versus general graphs,
\emph{J. Combin. Theory Ser. B} \textbf{173} (2025), 102--145.

\bibitem{MorrisICM}
R. Morris,
Some recent results in Ramsey theory,
in S. Friedlander, Y. Tschinkel (eds.),
\emph{Proceedings of the International Congress of Mathematicians
2026, Vol.~2: Plenary Lectures},
SIAM, 2026, 210--239.

\bibitem{Nikiforov}
V. Nikiforov,
The cycle-complete graph Ramsey numbers,
\emph{Combin. Probab. Comput.} \textbf{14} (2005), 349--370.

\bibitem{NikiforovRousseau}
V. Nikiforov, C. C. Rousseau,
Ramsey goodness and beyond,
\emph{Combinatorica} \textbf{29} (2009), 227--262.

\bibitem{PSPaths}
A. Pokrovskiy, B. Sudakov,
Ramsey goodness of paths,
\emph{J. Combin. Theory Ser. B} \textbf{122} (2017), 384--390.

\bibitem{PS}
A. Pokrovskiy, B. Sudakov,
Ramsey goodness of cycles,
\emph{SIAM J. Discrete Math.} \textbf{34} (2020), 1884--1908.

\bibitem{Ramsey}
F. P. Ramsey,
On a problem of formal logic,
\emph{Proc. Lond. Math. Soc.} \textbf{30} (1930), 264--286.

\bibitem{RostaI}
V. Rosta,
On a Ramsey-type problem of J. A. Bondy and P. Erd\H{o}s. I,
\emph{J. Combin. Theory Ser. B} \textbf{15} (1973), 94--104.

\bibitem{RostaII}
V. Rosta,
On a Ramsey-type problem of J. A. Bondy and P. Erd\H{o}s. II,
\emph{J. Combin. Theory Ser. B} \textbf{15} (1973), 105--120.
\end{thebibliography}
\end{document}